\documentclass[11pt,a4paper]{article}
\usepackage[dvips]{graphicx}
\usepackage[ansinew]{inputenc}
\usepackage{enumerate}
\usepackage[spanish,english]{babel}
\usepackage{amsthm,amssymb,amsfonts,amsmath,amsbsy}
\usepackage[ruled,linesnumbered]{algorithm2e}
\usepackage{color}

\newcounter{ab}

\newtheorem{theorem}{Theorem}[section]
\newtheorem{lemma}[theorem]{Lemma}
\newtheorem{proposition}[theorem]{Proposition}
\newtheorem{corollary}[theorem]{Corollary}

\newtheorem{conjecture}{Conjecture}
\newtheorem{claim}[ab]{Claim}

\newtheorem{remark}[theorem]{Remark}

\DeclareMathOperator{\aut}{Aut}
\DeclareMathOperator{\dete}{Det}
\DeclareMathOperator{\stab}{Stab}

\date{}

\title{Resolving sets for breaking symmetries of graphs}

\author{Delia Garijo
  \and Antonio Gonz\'alez
  \and Alberto M\'arquez }

\begin{document}
\maketitle

\begin{abstract}
This paper deals with the maximum value of the difference between the determining number and the
metric dimension of a graph as a function of its order. Our
technique requires to use locating-dominating sets, and perform an independent study
on other functions related to these sets. Thus, we obtain lower and upper bounds
on all these functions by means of very diverse tools. Among them are
 some adequate constructions of graphs, a variant of a classical result
in graph domination and a polynomial time algorithm  that produces both
distinguishing sets and determining sets. 
Further, we consider specific families of graphs where the restrictions of these functions
can be computed.
To this end, we utilize two well-known objects in graph theory:
$k$-dominating sets and matchings.

\end{abstract}

\section{Introduction and preliminaries}\label{sec:intro}

Every resolving parameter conveys  useful information about the behavior of distances in a graph.
 Thus, considering  several of   those  parameters together provides stronger properties of the underlying graph,
 which is the reason for studying the relations among them.
 Indeed, much effort has gone into relating metric dimension and other similar
 invariants including partition
 dimension~\cite{chappel,tomescu}, upper dimension~\cite{chartrandupper,MDUDRN},  and resolving
 number~\cite{chartrandchrom,jannesarion}, to name but a few.
 Combining metric dimension and  determining number
 allows us to obtain not only metric properties of   graphs but also an extra information about their symmetries.
 However, there are no many papers dealing with the connection between    these  two parameters: the determining number of a graph
is bounded above by its metric dimension \cite{boutin,erwin}. This prompts the following question posed by Boutin~\cite{boutin}: {\em Can the difference
between the determining number and the metric dimension of a graph 
 be arbitrarily large?}
  To deal with  this question, which is the main problem of this paper, we  next define   these  parameters.

Let $G$ be a connected graph\footnote{
Graphs in this paper are finite, undirected and simple. The vertex set and edge set of a graph $G$
are denoted by $V (G)$ and $E(G)$, respectively, and the order of $G$ is $n=|V(G)|$. We denote by $\overline{G}$ the complement of $G$.
 An {\em automorphism} of $G$ is a bijective mapping $f:V(G)\rightarrow V(G)$ such that $\{f(u),f(v)\}\in E(G)$
 if and only if $\{u,v\}\in E(G)$.
  The {\em automorphism group} of $G$ is written as $\aut (G)$, and its identity element is denoted by $id_G$. The {\it distance} $d_G(u,v)$ between two vertices $u$ and $v$ is the length of a
shortest $u$-$v$ path. We write $N_G(u)$ and $N_G[u]$ for the open and closed neighborhoods of any vertex $u\in V(G)$, respectively.
 Finally, $\delta_G (u)$ denotes the degree of $u$ and $\delta(G)$ is the  minimum degree  of $G$.
 We drop the subscript $G$ from these notations if the graph $G$ is clear from the context.}.
   Given a set $S\subseteq V(G)$, the {\it  stabilizer} of $S$ is $\stab (S) =\{ \phi\in \aut(G) : \phi(u) = u, \forall u\in S\}$,
 and $S$ is a {\em determining set} of $G$ if $\stab (S)$ is trivial.
The {\em determining number} of $G$, denoted by $\dete (G)$,
is the minimum cardinality of a determining set of $G$.
 A vertex $u\in V(G)$ {\it resolves} a pair $\{x,y\}\subseteq V(G)$ if $d(u,x)\neq d(u,y)$, and
   $S$ is a {\it resolving set} of $G$ if every pair of vertices of $G$ is resolved by some
vertex of $S$. The {\it metric dimension}
of $G$, written as $\dim (G)$, is the minimum order of a resolving set of $G$, and a resolving set of
size $\dim (G)$ is called a {\em metric basis} of $G$.

Determining sets were introduced in 2006 by Boutin \cite{boutin}, and independently by Erwin and Harary \cite{erwin},
 who adopted the term {\em fixing set}.
However, this concept was defined in a more general context in 1971 by Sims \cite{sims}:
a {\em base} of a permutation group of a set is a subset of elements whose stabilizer is trivial.
Also in the 1970s, resolving sets were introduced by Harary and Melter \cite{melter}, and
independently by Slater \cite{slater}.
These two types of sets have been widely studied in the literature because of
their multiple applications in very diverse areas. For instance,
 bases are useful tools for storing and analyzing large permutation groups \cite{blaha},
 and resolving sets are utilized for the graph isomorphism problem \cite{babai}.
 We refer the reader to the survey of Bailey and Cameron~\cite{bailey} for more references on these topics.

 As it was said before, there is a relationship between the determining number and the metric dimension: every resolving set
of a graph $G$ is also a determining set, and consequently $\dete (G)\leq \dim (G)$ \cite{boutin,erwin}.
 Let $(\dim-\dete)(n)$ be the maximum value of $\dim(G)-\dete(G)$ over all graphs $G$ of
order $n$. Thus, the computation of this function is equivalent to answer the above-mentioned question asked by Boutin~\cite{boutin}
about the difference between our parameters, which is widely studied by Cáceres et al.~\cite{cagapuse}.
 Namely, they provide the following bounds.
\begin{proposition} {\rm \cite{cagapuse} }\label{lowerCAGAPUSE} For every $n\geq 8$,
$$\lfloor\frac{2}{5}n\rfloor -2\leq ({\rm dim}-{\rm Det})(n)\leq n-2.$$
\end{proposition}

Fundamental to our technique,  which lets us
improve significatively the above result, are locating-dominating sets.
Hence,  we next introduce these sets
 together with the functions $(\lambda-{\rm Det})(n)$ and $\lambda(n)$,
for which we have to develop an independent study that is also of  interest.

A vertex $u\in V(G)$ {\em distinguishes} a pair $\{x,y\}\subseteq V(G)$ if either $u\in\{x,y\}$ or $N(x)\cap \{u\}\neq N(y)\cap \{u\}$,
 and a set $D\subseteq V(G)$ is a {\em distinguishing set} of $G$ if every pair of $V(G)$ is distinguished by some vertex of $D$. If $D$ is also a {\em dominating set} of $G$, i.e., $N(x)\cap D\neq \emptyset$ for every $x\in V(G)\setminus D$, then we say that $D$
 is a {\em locating-dominating set} of $G$.
 The minimum cardinality
 of a locating-dominating set  of $G$ is its {\em locating-domination number}, denoted by $\lambda (G)$.

Distinguishing sets were defined by Babai \cite{babai} when constructing canonical labelings
for the graph isomorphism problem, while Slater \cite{slaterLD} introduced locating-dominating
sets in the context of domination. However, these two concepts are in essence the same: one can easily check that
every distinguishing set becomes a dominating set  by adding at most one vertex. This implies the following observation.

\begin{remark} \label{sep-loc-dom} For any distinguishing set $D$ of a graph $G$, $\lambda(G)\leq |D|+1$.
\end{remark}

Every locating-dominating set of $G$ is clearly a resolving set, and so
${\rm Det} (G)\leq {\rm dim} (G)\leq \lambda (G)$ which leads us to pose
a similar question to that of Boutin~\cite{boutin} but concerning the difference
 $\lambda(G)-\dete(G)$.
Thus, let $(\lambda-{\rm Det})(n)$ and $\lambda(n)$ be the maximum values
of, respectively, $\lambda(G)-{\rm Det} (G)$ and $\lambda (G)$ over all graphs $G$ of order $n$.
 Although the function $\lambda(n)$ equals $n-1$ (just take the complete graph $K_n$), we need to define it
 because we shall consider a non-trivial restriction of $\lambda (n)$  which is quite useful throughout the paper. Therefore,
it is straightforward that
\begin{equation}\label{original}
({\rm dim}- {\rm Det})(n)\leq (\lambda-{\rm Det})(n)\leq \lambda(n)=n-1.
\end{equation}

 This paper undertakes a study on the function $({\rm dim}- {\rm Det})(n)$ which
 requires to develop a parallel study on  $(\lambda-\dete)(n)$ and the function
 $\lambda_{|_{\mathcal C^*}}(n)$ described in Section~\ref{sec:twin-free}.
 We thus start
  by constructing appropriate families of graphs which provide new lower
   bounds on $({\rm dim}- {\rm Det})(n)$ and $(\lambda-\dete)(n)$, 
  improving the lower bound of Proposition~\ref{lowerCAGAPUSE}  by Cáceres et al.~\cite{cagapuse}.
  Further, we conjecture that these are precisely the exact values of these functions. 
To improve the upper bound, we require a more sophisticated method which uses the locating-domination
number of twin-free graphs, namely the function $\lambda_{|_{\mathcal C^*}}(n)$.
Indeed, we first prove that this function is an upper bound on $(\dim-\dete)(n)$ and $(\lambda-\dete)(n)$,
 and  then  conjecture a presumable value of $\lambda_{|_{\mathcal C^*}}(n)$ which will be  supported through  the paper.

Subsequently, we obtain two explicit upper bounds on $\lambda_{|_{\mathcal C^*}}(n)$
  in Sections \ref{sec:ore} and \ref{sec:greedy}, respectively.
For the first one,  we give a different version of a classical theorem in domination theory due to Ore~\cite{ore}.
 This version leads us to a series of relationships between the locating-domination number
 and classical graph parameters in twin-free graphs,
 similar to the relations established among other domination parameters in many papers (see \cite{fod} for a number of examples). 
  Besides their own interest, these relations yield a first explicit bound on $\lambda_{|_{\mathcal C^*}}(n)$
   by using a nice Ramsey-type result due to Erd\H{o}s and Szekeres~\cite{szekeres}.

The second upper  bound that we provide on $\lambda_{|_{\mathcal C^*}}(n)$ is, until now, the best
 bound known on  $(\dim-\dete)(n)$.
 It is obtained from the greedy algorithm described in Section~\ref{sec:greedy} which
  produces both distinguishing sets and determining sets of bounded size in polynomial time.
  Hence, we also obtain a bound on the determining number of a twin-free graph.

Finally, we devote the last two sections  to the family of graphs
not containing the complete bipartite graph $K_{2,k}$ as a subgraph. Concretely, we
provide bounds and exact values of  our main functions  restricted  to
graphs without the cycle $C_4$ as a subgraph in Section~\ref{sec:C4}.
For this purpose, we obtain relationships between the locating-domination number of a twin-free graph
and other two well-known parameters: the $k$-domination number and the matching number.
Hence, we get bounds on these two invariants similar to other relations
  provided in a number of papers 
  (see Section~\ref{sec:C4} for the details).
  Furthermore, we compute the restrictions of $(\dim-\dete)(n)$ and $(\lambda-\dete)(n)$ to the family of
  trees in Section~\ref{sec:trees}, thereby closing
 the study initiated by Cáceres et al.~\cite{cagapuse} on this class of graphs.

\section{Lower bounds on $(\dim-\dete)(n)$ and $(\lambda-\dete)(n)$}\label{sec:lowerbound}

The question raised by Boutin~\cite{boutin}
    arose   from the fact that all graphs $G$ where
 she computed $\dim (G)-\dete(G)$ have a very small value of this difference. 
 Thus, Cáceres et al.~\cite{cagapuse} found a family of graphs with constant determining number and   metric dimension
 with linear growth:
  the wheel
graphs $W_{1,n}=K_1+C_n$ for which $\dim (W_{1,n})-\dete(W_{1,n})=\lfloor\frac{2}{5}n\rfloor -2$.
This implies a lower bound on the maximum value of this difference, i.e.,
a lower bound on the function $(\dim-\dete)(n)$ (see Proposition~\ref{lowerCAGAPUSE} above).
 In this section, we improve this   bound and also give a lower bound on $(\lambda-\dete)(n)$.
 To do this, we next provide two appropriate families of graphs.

For an integer $r\geq 6$, let $T_r$ be a path $(u_1,...,u_r)$ with a pendant vertex $u_0$ adjacent to $u_3$.
The {\em corona product} $G\circ K_1$ is the graph obtained from attaching a pendant vertex to every
vertex of any graph $G$. Let $G_r= T_r\circ K_1$ and let $H_r$ be the graph resulting from $G_r$ by attaching
a pendant vertex $v_0'$ to $u_0$ (see Figure~\ref{GrHr}).

\begin{figure}[ht]
\begin{center}
\begin{tabular}{cc}
\includegraphics[width=72mm]{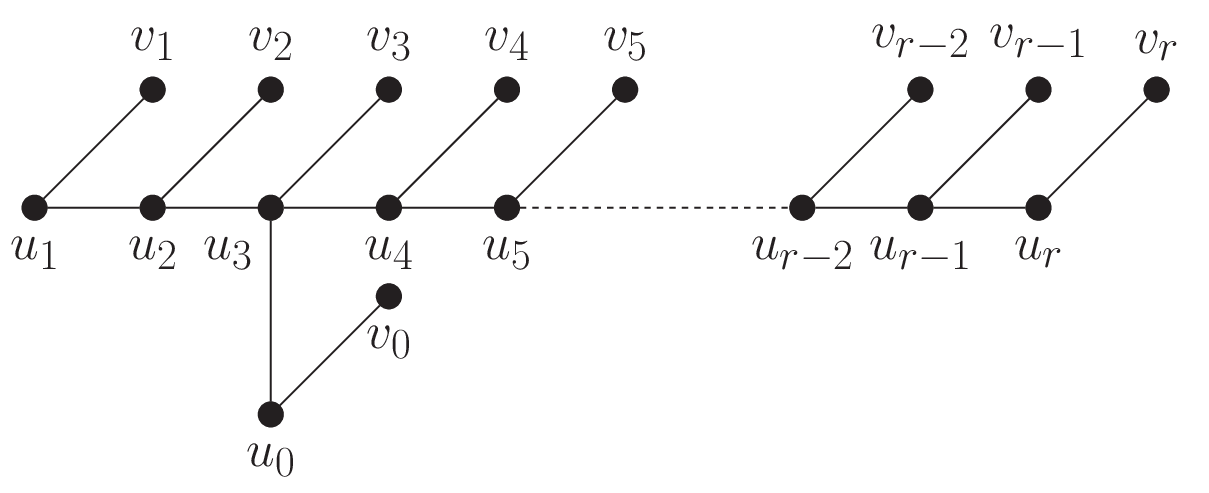}
&
\includegraphics[width=72mm]{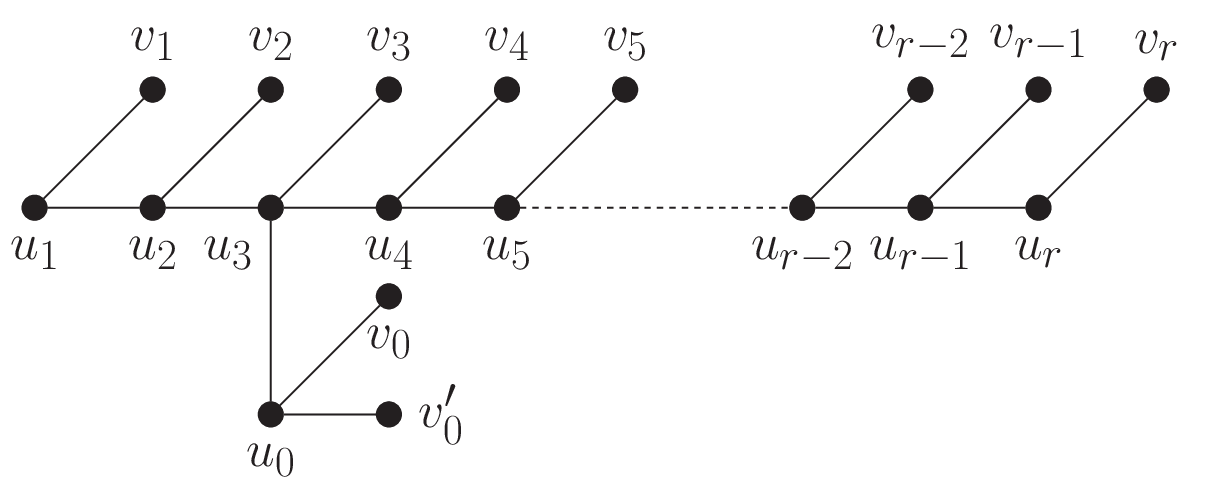}\\
(a)&(b)
\end{tabular}
\end{center}
\caption{The graphs (a) $G_r$ and (b) $H_r$.}\label{GrHr}
\end{figure}

The following lemma gives some evaluations of the main parameters considered
in the paper for the graphs $G_r$, $H_r$ and their complements $\overline{G}_r$ and $\overline{H}_r$. These are the key tools for proving Theorem~\ref{lowerbound} below.

\begin{lemma}\label{G_r}
For every $r\geq 6$, the following statements hold:
\begin{enumerate}
\item $\dete (G_r)=0$ and $\dete (H_r)=1$.
\item $\dim (\overline{G}_r)=r$ and $\dim (\overline{H}_r)=r+1$.
\item $\lambda (G_r)=r+1$ and $\lambda (H_r)=r+2$.
\end{enumerate}

\end{lemma}

\begin{proof}
Let $V(G_r)=\{u_0,...,u_r,v_0,...,v_r\}$ and $E(G_r)=E(T_r)\cup \{\{u_i,v_i\}:0\leq i\leq r\}$.
 Also, let $V(H_r)=V(G_r)\cup\{v_0'\}$ and $E(H_r)=E(G_r)\cup\{\{u_0,v_0'\}\}$. The three statements are proved one by one.

\begin{enumerate}
\item $\dete (G_r)=0$ since the automorphism group of $G_r$ is trivial.
On the other hand, $\aut(H_r)=\{id_{H_r},f\}$ where $f$ interchanges $v_0$ and $v_0'$, and fixes any other vertex.
Hence, $S=\{v_0\}$ is clearly a minimum determining set of $H_r$ and so $\dete (H_r)=1$.

\item  Observe that every resolving set $S$ of $\overline{G}_r$ contains either $u_i$ or $v_i$
 for every $0\leq i\leq r$, except for at most one. Otherwise, there are $u_i,v_i,u_j,v_j\not\in S$
 for some $i\neq j$, which implies that $d(u,v_i)=1=d(u,v_j)$ for every $u\in S$ and so
 $S$ is not a resolving set of $\overline G_r$; a contradiction. Therefore, $\dim(\overline{G}_r)\geq r$ and equality is given by the set  $S=\{u_0,u_1,...,u_{r-2},u_r\}$ which is a metric basis of $\overline{G}_r$.

The same arguments apply to prove that $S\cup\{v_0'\}$ is a metric basis of $\overline H_r$, and then $\dim(\overline{H}_r)=r+1$.

\item Let $D$ be a locating-dominating set of $G_r$. Note that either $u_i$ or $v_i$ belongs to $D$
for every $0\leq i\leq r$ (otherwise $N(v_i)\cap D=\emptyset$ and so $D$ is not a dominating set of $G_r$; a contradiction).
 Hence, $\lambda (G_r)\geq r+1$ and
equality holds since $D=\{u_0,...,u_r\}$ is clearly a locating-dominating set of $G_r$.

Arguing as above, we can check that $D\cup \{v_0'\}$ is a minimum locating-dominating set of $H_r$ and so $\lambda(H_r)=r+2$.
\end{enumerate}\end{proof}

With these values in hand, we next obtain lower bounds on $(\dim - \dete)(n)$ and
$(\lambda-\dete)(n)$ which in particular improve Proposition~\ref{lowerCAGAPUSE} above  due to Cáceres et al.~\cite{cagapuse}.

\begin{theorem}\label{lowerbound} For every $n\geq 14$,
$$(\dim-\dete)(n)\geq\lfloor\frac{n}{2}\rfloor -1\qquad and \qquad (\lambda-\dete)(n)\geq\lfloor\frac{n}{2}\rfloor .$$
\end{theorem}

\begin{proof} To prove that $(\dim-\dete)(n)\geq\lfloor\frac{n}{2}\rfloor -1$ we only need to show that, for every $n\geq 14$, there exists a graph $G$ of order $n$
such that $\dim(G)-\dete(G)=\lfloor\frac{n}{2}\rfloor-1$.
If $n$ is even then let $G=\overline{G}_{\frac{n}{2}-1}$. Indeed, $\overline{G}_{\frac{n}{2}-1}$ has order $n$ and  Lemma~\ref{G_r} yields  $\dim(\overline{G}_{\frac{n}{2}-1})-\dete(\overline{G}_{\frac{n}{2}-1})=\lfloor\frac{n}{2}\rfloor -1$ (note that  $\dete (G)=\dete (\overline{G})$ holds for any graph $G$ since $\aut(G)=\aut(\overline{G})$).
 Otherwise, $n$ is odd and we set $G=\overline{H}_{\frac{n-1}{2}-1}$, whose order is   $n$ and satisfies $\dim(\overline{H}_{\frac{n-1}{2}-1})-\dete(\overline{H}_{\frac{n-1}{2}-1})=\lfloor\frac{n}{2}\rfloor -1$, by Lemma~\ref{G_r}.

Considering the graphs $G_{\frac{n}{2}-1}$ (when $n$ is even) and $H_{\frac{n-1}{2}-1}$ (when $n$ is odd),
 one can use the same arguments as above to show that   $(\lambda-\dete)(n)\geq\lfloor\frac{n}{2}\rfloor$.
\end{proof}

In the remainder of the paper, we shall exhibit wide classes of graphs where the restrictions
of $(\dim-\dete)(n)$ and $(\lambda-\dete)(n)$ do not exceed $\frac{n}{2}$. Thus, there are reasons to believe that the
bounds given above are the exact values of these functions. We
state this as a conjecture.

\begin{conjecture} \label{centralconj} There exists a positive integer $n_0$ such that, for every $n\geq n_0$,
$$(\dim-\dete)(n)=\lfloor\frac{n}{2}\rfloor-1 \qquad and \qquad(\lambda-\dete)(n)=\lfloor\frac{n}{2}\rfloor.$$
\end{conjecture}

\section{An upper bound on $(\dim-\dete)(n)$ and $(\lambda-\dete)(n)$}\label{sec:twin-free}

In this section, we assemble  all the necessary machinery in order to prove
  that  $(\dim-\dete)(n)$ and $(\lambda-\dete)(n)$ are bounded above by the function $\lambda_{|_{\mathcal C^*}}(n)$
  which is the maximum value of $\lambda(G)$ over all twin-free graphs $G$ of order $n$ (see Theorem~\ref{lambda-deten}). 
 The proof requires basically the following two ideas: 
 the construction of a twin-free graph $\widetilde{G}$ from an arbitrary graph $G$ based on the so called twin graph described in Subsection~\ref{subsec:twingraph}, and the relationship between $\lambda(G)$ and $\lambda(\widetilde G)$ given in Subsection~\ref{subsec:tilde}.

\subsection{The twin graph $G^*$}\label{subsec:twingraph}

For a  graph $G$, the twin graph $G^*$ is obtained from $G$ by identifying vertices with the same neighborhood.
This construction and any of its variations (depending on the choice of closed and/or open neighborhoods)
completely characterize the original graph, which is the reason why they
have been considered for solving many problems
in graph theory  (see for instance \cite{hatami,heggernes,extremal,roberts}).
In this subsection, we provide some properties of  $G^*$ and a bound on $\dete (G)$ in terms of $|V(G^*)|$,
each of which is useful in the paper. The twin graph is formally defined as follows.

 Two different vertices $u,v\in V(G)$ are {\it twins} if $N(u)=N(v)$ or $N[u]=N[v]$, 
  i.e., no vertex of $V(G)\setminus\{u,v\}$
distinguishes $\{u,v\}$.
 It is proved in \cite{extremal} that this definition induces an equivalence
  relation   on $V(G)$ given by $u\equiv v$ if and only if either $u=v$ or $u$ and $v$
   are twins. Thus, let $u^*=\{v\in V(G): u\equiv v\}$
   and consider the partition $u_1^*,...,u_r^*$ of $V(G)$ induced by this relation, where every $u_i$ is a representative of $u_i^*$.
 The  {\em twin graph} of $G$, written as $G^*$, has vertex set $V(G^*)=\{u_1^*,...,u_r^*\}$ and edge set $E(G^*)=\{\{u_i^*,u_j^*\}:\{u_i,u_j\}\in E(G)\}$.
Note that, for every $x\in V(G)$, we shall consider $x^*$ as a class in $V(G)$, as well as a vertex of $G^*$.
 The twin graph $G^*$ has the following properties.

\begin{lemma}{\rm\cite{extremal}}\label{lemmaextremal} For every graph $G$, the following statements hold:
\begin{enumerate}
\item \label{edges} The   graph $G^*$ is independent of the choice of the representatives $u_i$, i.e., 
 $$\{u_i^*,u_j^*\}\in E(G^*) \qquad\Longleftrightarrow \qquad\{x,y\}\in E(G) \quad \forall x\in u_i^*, y\in u_j^*$$

\item \label{completeindependent} Every class $u_i^*$ either induces a complete subgraph or is an independent set in $G$.
\end{enumerate}
\end{lemma}

  A vertex $u_i^*\in V(G^*)$ is of {\em type $(1)$} if $|u_i^*|=1$; otherwise  $u_i^*$ is of {\em type} $(KN)$.
  According to Statement~\ref{completeindependent} of Lemma~\ref{lemmaextremal}, a vertex $u_i^*$ of type
  $(KN)$ is of {\em type} $(K)$ or $(N)$, depending on weather $u_i^*$ induces a complete subgraph or
  is an independent set in $G$. Note that $N[x]=N[y]$ for every $x,y\in u_i^*$ whenever $u_i^*$ is of type $(K)$, and
  $N(x)=N(y)$ for every $x,y\in u_i^*$ whenever $u_i^*$ is of type $(N)$.
 For more properties  of $G^*$ we refer the reader to \cite{extremal}.

Now, we give two lemmas  considering the twin graph $G^*$ which will be helpful in the remainder of the paper.

\begin{lemma}\label{class1} For every graph $G$,
 no two different vertices $u_i^*,u_j^*\in V(G^*)$ of type $(1)$ are twins in $G^*$.
\end{lemma}
\begin{proof}
Since $u_i^*\neq u_j^*$ there exists a vertex $u\in V(G)\setminus\{u_i,u_j\}$ distinguishing $\{u_i,u_j\}$. Without loss of generality,
assume that  $u_i\in N_G(u)$ and  $u_j\not\in N_G(u)$. Also, observe that
$u$ is not contained in $u_i^*$ or $u_j^*$
 since
they are of type $(1)$. Therefore, $u_i^*\in N_{G^*}(u^*)$ and  $u_j^*\not\in N_{G^*}(u^*)$, which implies that $u_i^*$ and $u_j^*$ are not twins in $G^*$.
\end{proof}

Let $\Omega_G=\bigcup_{1\leq i\leq r}u_i^*\setminus\{u_i\}$ which is composed by all
but one vertex of every class of type $(KN)$. Clearly, this set has cardinality $n-r$ and satisfies
 that no two vertices of $V(G)\setminus \Omega_G$ are twins in $G$. Observe
 that  $\Omega_G$ is also independent of the choice of the representatives $u_i$.
   Using this set, we can prove the following result.

\begin{lemma}\label{detn-r}
Let $G$ be a graph of order $n$ such that $G^*$ has order $r$. Then,
   $$\dete (G)\geq n-r.$$
   In particular,  $\lambda(G)-\dete(G)\leq r-1$.
\end{lemma}

\begin{proof}
Let $u_i^*$ be a class of type $(KN)$ in $V(G)$. For each $x,y\in u_i^*$, let $f\in\aut (G)$
 fixing every vertex  of $G$ but $x$ and $y$, which are interchanged.
 Obviously, $f\in\stab (V(G)\setminus\{x,y\})$ and $f\neq id_G$.
Hence, every determining set $S$ of $G$ contains either $x$ or $y$.
It follows that $S$ contains all but one vertex of each
 class of type (KN), i.e., $|\Omega_G|=n-r$ vertices. Therefore, $\dete (G)\geq n-r$ and combining this with $\lambda (G)\leq n-1$ yields
  $\lambda(G)-\dete(G)\leq r-1$.
\end{proof}

\subsection{Using locating-dominating sets of twin-free graphs}\label{subsec:tilde}

In this subsection, we provide an upper bound on the functions $(\dim-\dete)(n)$ and $(\lambda-\dete)(n)$ based
on the locating-domination number of twin-free graphs.
 These graphs are important for their own sake \cite{auger,augercharon,charon1,sumner}, and also for their many
applications to other problems in graph theory  \cite{foucaudsylvain,hatami,kotlov}. Here, we
 construct a twin-free graph $\widetilde G$ for every graph $G$ (whenever $G^*\not\cong K_2$) in such a way that we can obtain
 locating-dominating sets of $G$ from those of $\widetilde G$.
 This construction and the relation between $\lambda (G)$ and $\lambda (\widetilde G)$ given in Lemma~\ref{lambdatilde} below are the
 key tools to prove  the above-mentioned   bound on $(\dim-\dete)(n)$ and
 $(\lambda-\dete)(n)$ (see Theorem~\ref{lambda-deten}).

A graph $G$ is {\it twin-free} if it does not contain twin vertices.
Observe that $G^*$ is not necessarily twin-free (see for instance Figures~\ref{tildeG}(a) and \ref{tildeG}(b)).
 However, we shall use this graph to associate a twin-free graph $\widetilde G$ to $G$.
 Indeed, let $\widetilde{G}$ be the graph obtained from $G^*$ by attaching a pendant vertex
to every $u_i^*\in V(G^*)$ of type $(KN)$ whenever $u_i^*$ has some twin in $G^*$ (see Figure~\ref{tildeG}(c)).
Thus, let us denote $V(\widetilde{G})=V^*\cup \mathcal{L}$, where $V^*=\{u_1^*,...,u_r^*\}$ and $\mathcal{L}$ is
the set of pendant vertices adjacent to vertices of $V^*$.
Note again that, for every $x\in V(G)$,  $x^*$ denotes  a class in $V(G)$, a vertex of $V(G^*)$, as well as
a vertex of $V^*\subseteq V(\widetilde G)$.
\\

\begin{figure}[ht]
\begin{center}
\begin{tabular}{ccccccccc}
\includegraphics[width=35mm]{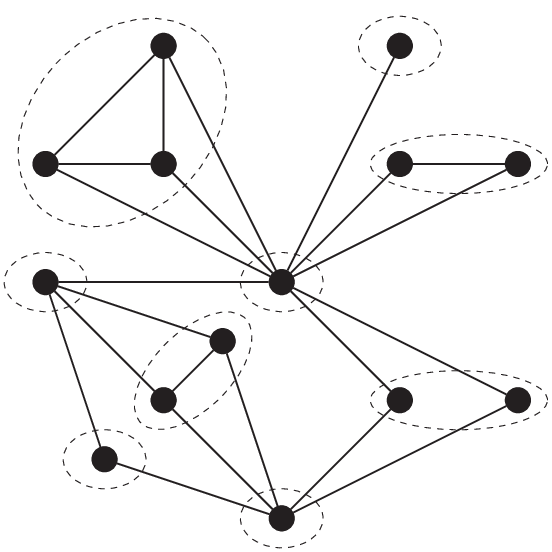}&&&&
\includegraphics[width=35mm]{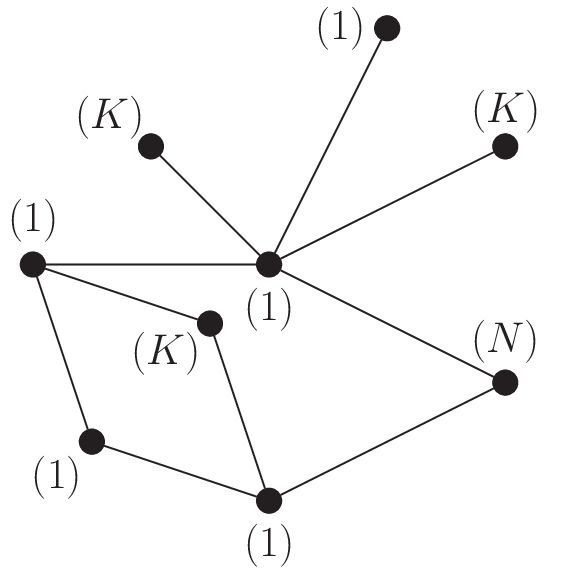}&&&&
\includegraphics[width=35mm]{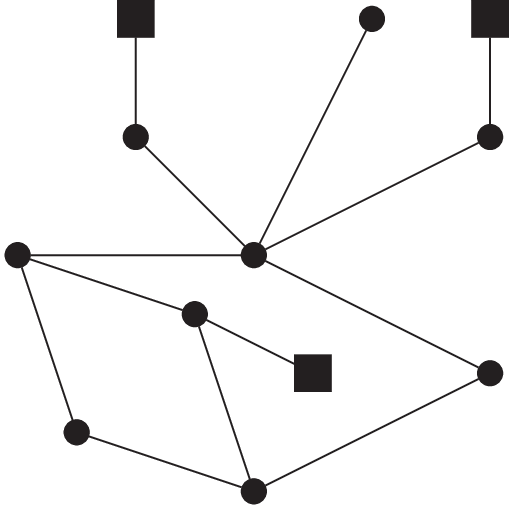}\\
(a)&&&&(b)&&&&(c)\\
\end{tabular}
\end{center}
\caption{(a) A graph $G$ and its classes in $G^*$, (b) the twin graph $G^*$ and (c) the associated graph $\widetilde G$.}\label{tildeG}
\end{figure}

We next provide two technical lemmas about $\widetilde G$ which
 are useful in the proof of Theorem~\ref{lambda-deten} and other proofs of this paper.

\begin{lemma}\label{Gtilde}
Let $G$ be a graph of order $n$ such that $G^*$ is not isomorphic to $K_2$. Then, the graph $\widetilde{G}$ has order $\widetilde{n}\leq n$ and is twin-free.
\end{lemma}

\begin{proof}

When obtaining  $G^*$ from $G$, we "lose" at least one vertex for each class of type $(KN)$ since they contain at least two vertices.
Thus, every time we attach a pendant vertex to a vertex of type $(KN)$ for constructing $\widetilde{G}$ from $G^*$,
we do not exceed the order of $G$.
Hence, $\widetilde n\leq n$.

Now, we prove that $\widetilde{G}$ is twin-free. On the contrary, suppose that $\widetilde{G}$ has a pair of twins, say $u,v\in V(\widetilde{G})$. 
If $u,v\in V^*$ then it is easy to see that they are also twins in $G^*$.
 Hence, at least one of them is of type $(KN)$, by Lemma~\ref{class1},
 and so they are distinguished in $\widetilde{G}$ by the corresponding pendant vertex  of $\mathcal{L}$; a contradiction.
Moreover, no two pendant vertices of $\mathcal{L}$ are twins since they are
 adjacent to different vertices of $V^*$. Therefore, we can
  assume that $u\in V^*$ and $v\in\mathcal{L}$.

Let $u=u_i^*$  and let $u_j^*$ be such that $N_{\widetilde G}(v)=\{u_j^*\}$ for some $1\leq i,j\leq r$.
Since $u$ and $v$ are twins, we can assume that $N_{\widetilde G}(u)=N_{\widetilde G}(v)=\{u_j^*\}$ 
(otherwise $N_{\widetilde G}[u]=N_{\widetilde G}[v]=\{v,u_j^*\}$
and so  $\widetilde G\cong K_2$ since $\widetilde G$ must be connected, which implies
that $G^*\cong K_1$; a contradiction since $G^*\cong K_1$ implies that $\widetilde G\cong K_1$).
 Hence, $u_i^*\neq u_j^*$ and,  
  by construction of $\widetilde G$, $u_j^*$ must have a twin in $G^*$, say $u_k^*$. Clearly, $u_k^*\neq u_i^*$ (otherwise
  $u_i^*$ and $u_j^*$ are twins in $G^*$ and so $G^*\cong K_2$ since $N_{\widetilde G}(u_i^*)=N_{G^*}(u_i^*)=\{u_j^*\}$; a contradiction).
  Thus, $u_k^*\in N_{G^*}(u_i^*)$ since $u_j^*$ and $u_k^*$ are twins
   and we know that  $u_j^*\in N_{G^*}(u_i^*)$ but $N_{\widetilde G}(u^*_i)=N_{G^*}(u_i^*)=\{u_j^*\}$, which is a contradiction.
\end{proof}

The following lemma establishes a relationship  between $\lambda(G)$ and $\lambda(\widetilde G)$, and is a key result for proving Theorem~\ref{lambda-deten}.

\begin{lemma}\label{lambdatilde}
Let $G$ be a graph of order $n$ such that $G^*$ has order $r$. Then,
$$\lambda (G)\leq \lambda (\widetilde{G}) + n-r.$$
 In particular, $\lambda(G)-\dete(G)\leq\lambda(\widetilde G)$.
\end{lemma}

\begin{proof}

Let $S\subseteq V(\widetilde{G})$ be a minimum locating-dominating set of $\widetilde{G}$.
 Observe that, for every $u\in V(\widetilde{G})$, there is a unique $u_i\in V(G)$ with $1\leq i\leq r$ such that either
 $u=u_i^*$ or $N_{\widetilde{G}}(u)=\{u_i^*\}$ (depending on whether $u\in V^*$ or $u\in\mathcal{L}$).
 Thus, let $\pi(u)$ be such  $u_i$ and $\pi(S)=\{\pi(u):u\in S\}$.
Clearly, the set $\pi(S)$ satisfies $|\pi(S)|\leq|S|=\lambda(\widetilde{G})$ because $\pi(u)=\pi(v)$ whenever $u,v\in S$
 with $u\in V^*$, $v\in \mathcal{L}$ and $N_{\widetilde{G}}(v)=\{u\}$.
 Therefore,  we will obtain the expected bound by proving that
$ S'=\pi(S)\cup\Omega_G$ is a locating-dominating set of $G$ since $|\Omega_G|=n-r$.

First, observe that $\pi (u)=u_i$ implies that $u_i^*\subseteq S'$ whenever $u\in S$.
 We claim that $S'$ is a distinguishing set of $G$. Indeed,
given $x,y\in V(G)\setminus S'$, we shall prove that $\{x,y\}$ is distinguished by some vertex of $S'$.
Obviously, we can assume that $x^*\neq y^*$ in $G^*$
(otherwise either $x$ or $y$ belongs to $\Omega_G\subseteq S'$; a contradiction since $x,y\in V(G)\setminus S'$). Since $S$ is a locating-dominating set
of $\widetilde{G}$, then there is $u\in S$ distinguishing $\{x^*,y^*\}$ in $\widetilde{G}$, and we can suppose that $u\neq x^*\neq y^*$  (otherwise either $u=x^*\subseteq S'$ or $u=y^*\subseteq S'$; a contradiction).

  If $u\in \mathcal{L}$ then, without loss of generality, let us assume that $x^*\in N_{\widetilde G}(u)$ and $y^*\not\in N_{\widetilde G}(u)$. Thus,
  $\pi(u)=u_i$ with $u_i^*=x^*$ and  so $x^*\subseteq S'$; a contradiction with $x\in V(G)\setminus S'$. Hence, $u\in V^*$ and so
  $u=u_i^*$ for some $1\leq i\leq r$, which leads to $\pi (u)=u_i\in S'$. Assuming that $x^*\in N_{\widetilde G}(u)$ and
  $y^*\not\in N_{\widetilde{G}}(u)$ (the opposite case is similar), we have that $x\in N_G(u_i)$ and $y\not\in N_G(u_i)$,
  by Statement~\ref{edges} of Lemma~\ref{lemmaextremal},  which implies that $u_i$ distinguishes $\{x,y\}$.
Therefore, $S'$ is a distinguishing set of $G$ and a similar analysis  shows that it is also a dominating set. 

We have just proved that $\lambda (G)\leq\lambda(\widetilde G)+n-r$, which combined with $\dete (G)\geq n-r$ (see Lemma~\ref{detn-r})
yields $\lambda(G)-\dete(G)\leq\lambda(\widetilde G)$, as claimed.
\end{proof}

 For any class of graphs $\mathcal{C}$, we define $(\dim-\dete)_{|_{\mathcal{C}}}(n)$, $(\lambda-\dete)_{|_{\mathcal{C}}} (n)$ and
 $\lambda_{|_{\mathcal{C}}} (n)$ as in Section~\ref{sec:intro} but restricting their domains to the graphs of $\mathcal{C}$.
Let $\mathcal{C}^*$ be the class of twin-free graphs. Thus, the function $\lambda_{|_{\mathcal C^*}}(n)$
can be considered for every $n\geq 4$ since   $P_4$ is clearly the smallest twin-free graph.

We now reach the main result of this section which improves significatively Expression~(\ref{original}).

\begin{theorem}\label{lambda-deten}
 For every $n\geq 4$,
$$(\dim-\dete) (n) \leq(\lambda-\dete) (n) \leq\lambda_{|_{\mathcal C^*}}(n).$$
\end{theorem}

\begin{proof}
Since the first inequality is obvious, we only need to show that $(\lambda-\dete) (n) \leq\lambda_{|_{\mathcal C^*}}(n)$.
We begin by proving the following two claims.

\begin{claim}\label{add_pendant} For a graph $G$, let $H$ be the graph obtained from $G$ by attaching a pendant vertex $u$ to a vertex $v\in V(G)$.
Then, $\lambda (G)\leq \lambda (H)$.
\end{claim}
\begin{proof}
Let $S\subseteq V(H)$ be a minimum locating-dominating set of $H$. Clearly, if $u\notin S$ then $S\subseteq V(G)$ is also a locating-dominating set
of $G$, and so $\lambda (G)\leq \lambda (H)$. Otherwise $u\in S$ and it is easy to check that $S'=(S\setminus \{u\})\cup \{v\}$ is a locating-dominating set
of $G$. 
 Therefore, $\lambda (G)\leq |S'|\leq\lambda (H)$.
\end{proof}

\begin{claim}\label{monotone} $\lambda_{|_{\mathcal C^*}}(n)\leq \lambda_{|_{\mathcal C^*}}(n+1).$
\end{claim}

\begin{proof}
Consider a twin-free graph $G$ of order $n$ such that $\lambda (G)=\lambda_{|_{\mathcal C^*}}(n)$. To prove the claim,
it suffices to find a twin-free graph $H$ of order $n+1$ such that  $\lambda (H)\geq \lambda (G)$.
 Indeed, let $H$ be the graph  obtained from $G$ by attaching a pendant vertex $u$ to a vertex $v\in V(G)$ such that
 no neighbor of $v$ has degree 1 in $G$. Note that this is possible since $G$ is not the disjoint union of copies of $K_1$ or $K_2$,
which is neither connected nor twin-free. Hence, $H$ has order $n+1$ and is clearly twin-free. Moreover,
 Claim~\ref{add_pendant} ensures that $\lambda (H)\geq \lambda (G)$, as required.
\end{proof}

Now, we are able to prove the theorem. Thus, let $G$ be a graph of order $n$ such that $\lambda(G)-\dete(G)=(\lambda-\dete)(n)$.
 Lemma~\ref{lambdatilde} yields
\begin{equation}\label{eq1}
(\lambda-\dete)(n)=\lambda(G)-\dete(G)\leq \lambda (\widetilde{G}).
\end{equation} 
 On the other hand,
 if $G^*\cong K_2$ then, by Lemma~\ref{detn-r}, we have $(\lambda-\dete)(n)=\lambda(G)-\dete(G)\leq 1<\frac{n}{2}$;
a contradiction with Theorem~\ref{lowerbound}.
Hence, $G^*\not\cong K_2$ and so
 Lemma~\ref{Gtilde} says that $\widetilde G$ is twin-free and $\widetilde n=|V(\widetilde G)|\leq n$.
 Thus, we get
\begin{equation}\label{eq2}
\lambda (\widetilde G)\leq \lambda_{|_{\mathcal C^*}}(\widetilde n)\leq \lambda_{|_{\mathcal C^*}}(n),
\end{equation}
 the last inequality being a consequence of Claim~\ref{monotone}. Therefore, combining Expressions (\ref{eq1}) and
 (\ref{eq2}) gives the expected inequality.
\end{proof}

The preceding theorem implies that bounding the function $\lambda_{|_{\mathcal C^*}}(n)$ yields bounds on
$(\dim-\dete) (n)$ and $(\lambda-\dete) (n)$. Thus, we will be mainly concerned with the locating-domination number
 of twin-free graphs in the following two sections.

 Theorems \ref{lowerbound} and \ref{lambda-deten} give $\lambda_{|_{\mathcal C^*}}(n)\geq\lfloor\frac{n}{2}\rfloor$
 and, throughout this paper, we shall find  numerous conditions for a twin-free graph to satisfy $\lambda(G)\leq \frac{n}{2}$.
Thus, we believe that   the
following conjecture, which implies most of Conjecture~\ref{centralconj}, is true.

\begin{conjecture}\label{twin_free_conj}
There exists a positive integer $n_1$ such that, for every $n\geq n_1$,
$$\lambda_{|_{\mathcal C^*}}(n)=\lfloor\frac{n}{2}\rfloor.$$
\end{conjecture}

\section{From minimal dominating sets to locating-dominating sets}\label{sec:ore}

In this section, we present a variant of a theorem by Ore~\cite{ore} in domination theory which leads us to a  bound
 on $\lambda_{|_{\mathcal C^*}}(n)$ (see Corollary~\ref{erdos_coro}) by means of a classical result due to Erd\H{o}s and Szekeres~\cite{szekeres}. 
 Further, this variant allows us to relate the locating-domination number of a twin-free graph to classical graph parameters: upper domination number,
independence number, clique number and chromatic number. All these relations produce a number of sufficient conditions
for a twin-free graph $G$ to verify $\lambda(G)\leq \frac{n}{2}$, i.e.,
 they support Conjecture~\ref{twin_free_conj} in numerous cases (see Corollaries~\ref{Gamma}, \ref{alphaomega} and \ref{chi}).

A set $S\subseteq V(G)$ is a {\em minimal  dominating set}  if no proper subset of $S$ is a  dominating set of $G$
({\em minimal locating-dominating sets} are defined analogously). The following theorem  due to Ore~\cite{ore}
is one of the first results in the field of domination, which is an area that has played a central role in graph theory for the last fifty years.
 We refer the reader to \cite{fod} for an extensive bibliography on domination related concepts.

\begin{theorem}\label{complementminimal}
{\rm \cite{ore}} Let $G$ be a graph without isolated vertices and let $D\subseteq V(G)$ be a minimal dominating set of $G$.
Then, $V(G)\setminus D$ is a dominating set of $G$. Consequently, $\gamma(G)\leq\frac{n}{2}$.
\end{theorem}

Observe that an analogue of this last result but for locating-dominating sets of twin-free graphs would prove in the affirmative Conjecture~\ref{twin_free_conj}.
 Unfortunately, the complement of a minimal locating-dominating set of a twin-free graph is not necessarily a locating-dominating set,
as shown in Figure~\ref{counterexample}.
However, we next provide a similar relation between minimal dominating sets and locating-dominating sets
which improves Theorem~\ref{complementminimal} in the twin-free case.

\begin{figure}[ht]
\begin{center}
\includegraphics[width=40mm]{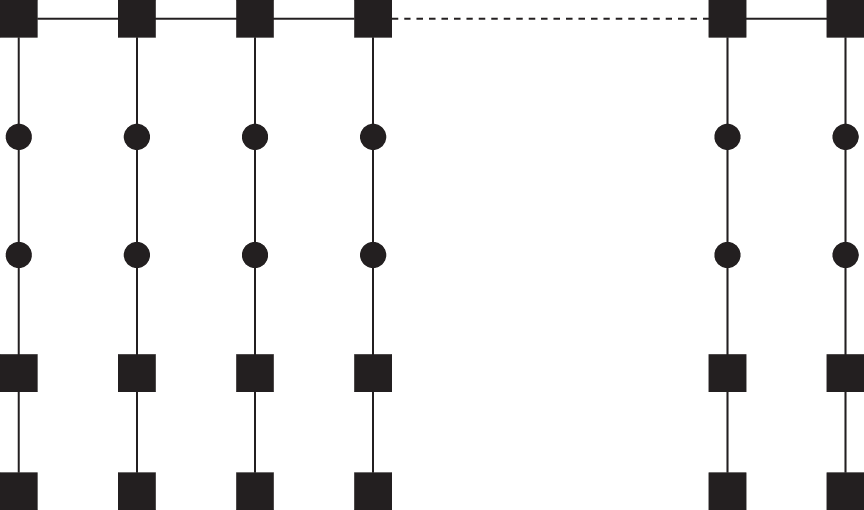}
\caption{A twin-free graph with a minimal locating-dominating set (depicted as square vertices) whose complement is not a locating-dominating set.}\label{counterexample}
\end{center}
\end{figure}

\begin{theorem}\label{LD_ore-type}
Let $G$ be a twin-free graph and let $D\subseteq V(G)$ be a minimal dominating set of $G$. Then, $V(G)\setminus D$ is a locating-dominating set of $G$.
\end{theorem}

\begin{proof}
Let $D$ be a minimal dominating set of $G$. By Theorem~\ref{complementminimal}, we only need to prove that $V(G)\setminus D$ is a distinguishing set of $G$. Thus, we shall show that, for every
 $x,y\in D$, there is a vertex $u\in V(G)\setminus D$ distinguishing $\{x,y\}$.
Indeed, it is proved in \cite{ore} that  $D$ is a minimal dominating set if and only if each vertex $x\in D$ satisfies that either $N(x)\subseteq V(G)\setminus D$ or
$N(u)\cap D = \{x\}$ for some $u\in V (G) \setminus D$.
 Hence, if $N(x),N(y)\subseteq V(G)\setminus D$ then $\{x,y\}$ is distinguished by some $u\in V(G)\setminus D$ since $G$ is twin-free.
 Otherwise, we can assume  without loss of generality that $N(u)\cap D=\{x\}$ for some $u\in V(G)\setminus D$ and so $\{x,y\}$ is distinguished by $u$. Therefore, $V(G)\setminus D$ is a locating-dominating set of $G$, as claimed.
\end{proof}

Now, we show a series of consequences of this last result which relate $\lambda(G)$ to well-known
graph parameters when $G$ is twin-free.
The {\em upper domination number} $\Gamma (G)$ of a graph $G$ is the maximum
cardinality of a minimal dominating set of $G$.
 This is a heavily studied invariant which has been related to other well-known  parameters in the area of domination (see \cite{fod} for multiple examples).
  With the same spirit, the following consequence of Theorem~\ref{LD_ore-type} relates the upper domination number
to the locating-domination number of a twin-free graph,
and supports the validity of Conjecture~\ref{twin_free_conj}.

\begin{corollary}\label{Gamma}
Let $G$ be a twin-free graph. Then, $$\lambda(G)\leq n-\max{\{\Gamma(G),\Gamma(\overline{G})-1\}}.$$ In particular,
$\lambda (G)\leq \frac{n}{2}$ when either $\Gamma(G)\geq \frac{n}{2}$ or $\Gamma (\overline{G})\geq \frac{n}{2}+1$.
\end{corollary}

\begin{proof}
We deduce from Theorem~\ref{LD_ore-type} that $\lambda(G)\leq n-\Gamma(G)$ for every twin-free graph $G$.
Also, Theorem 7 of \cite{nordhaus-gaddum}
shows that $|\lambda(G)-\lambda(\overline{G})|\leq 1$   and so $\lambda(G)\leq \lambda(\overline G)+1\leq n-\Gamma(\overline G)+1$
since $\overline G$ is also twin-free.
Therefore, $\lambda(G)\leq \min\{n-\Gamma(G),n-\Gamma(\overline G)+1\}$, which is
the expected bound.
\end{proof}

Recall that the {\em independence number} $\alpha(G)$ and  the {\em clique number} $\omega(G)$ are
the maximum cardinalities of  an independent set and a complete subgraph of $G$, respectively.
The following result relates these two classical parameters to $\lambda(G)$ when $G$ is twin-free,
 and gives another sufficient condition for $G$
 to have $\lambda(G)\leq \frac{n}{2}$.

\begin{corollary}\label{alphaomega}

Let $G$ be  a twin-free graph. Then, $$\lambda (G)\leq n -\max \{\alpha (G),\omega (G)-1\}.$$ In particular,
$\lambda (G)\leq \frac{n}{2}$ when either $\alpha (G)\geq \frac{n}{2}$ or $\omega (G)\geq \frac{n}{2}+1$.
\end{corollary}

\begin{proof}

Observe first that every independent set $I$ of order $\alpha(G)$ is a minimal
dominating set of $G$. Indeed, $I$ is a dominating set since every $x\in V(G)\setminus I$ has a neighbor in $I$ (otherwise
$I$ is not an independent set of maximum order) and it is minimal since $N(u)\subseteq V(G)\setminus I$ for every $u\in I$.
 Hence, $\alpha(G)\leq\Gamma(G)$, and so $\omega(G)=\alpha(\overline{G})\leq\Gamma(\overline G)$. Thus, combining these inequalities with
 Corollary~\ref{Gamma} leads us to the bound since $G$ is twin-free and so is $\overline G$.
\end{proof}

The {\em chromatic number} of $G$, denoted by $\chi(G)$, is the smallest number of classes needed to partition $V(G)$
 so that no two adjacent vertices belong to the same class. A classical result in graph theory establishes that $\chi (G)\leq  \frac{n+\omega(G)}{2}$ (see for instance \cite{chromaticGT}).  Applying this  to $G$ and $\overline{G}$, we can easily deduce from Corollary~\ref{alphaomega}
 the following bound on $\lambda(G)$ in terms of $\chi(G)$ and $\chi(\overline G)$ which in particular supports
  Conjecture~\ref{twin_free_conj}.

\begin{corollary}\label{chi}
Let $G$ be  a twin-free graph. Then, $$\lambda (G)\leq 2n -\max \{2\chi(G),2\chi(\overline G)-1\}.$$ Consequently,
 $\lambda (G)\leq \frac{n}{2}$ when either $\chi (G)\geq \frac{3}{4}n$ or $\chi (\overline{G})\geq \frac{3}{4}n+\frac{1}{2}$.
\end{corollary}

Erd\H{o}s and Szekeres~\cite{szekeres} proved that every graph of order $n$ contains either a complete subgraph or an
 independent set of cardinality at least $\lceil \frac{\log_2n}{2}\rceil$. On account of this result and Corollary~\ref{alphaomega},
 we obtain our first upper bound on $\lambda_{|_{\mathcal C^*}}(n)$, and
 consequently on $(\dim -\dete)(n)$ and $(\lambda-\dete)(n)$ (by Theorem~\ref{lambda-deten}). Thus, the following corollary
  improves significatively
 the   bound $(\dim -\dete)(n)\leq n-2$ due to Cáceres et al.~\cite{cagapuse} (see Proposition~\ref{lowerCAGAPUSE} above).
\begin{corollary}\label{erdos_coro} For every $n\geq 4$,
 $$(\dim -\dete)(n)\leq(\lambda-\dete)(n)\leq  \lambda_{|_{\mathcal C^*}}(n) \leq n -\lceil\frac{\log_2n}{2}\rceil+1.$$
\end{corollary}

\section{A greedy algorithm for finding distinguishing sets and determining sets of twin-free graphs}\label{sec:greedy}

Babai \cite{babai} defined distinguishing sets because of their usefulness in the graph isomorphism problem.
Indeed, by constructing a canonical labeling, he proved that deciding whether a graph $G$ of order $n$ is isomorphic to any other can be done in  $o(n^{s+3})$ time
whenever $G$ has a distinguishing set of size $s$.
 Thus, Babai provided the following result on distinguishing sets by means of a probabilistic argument.

\begin{lemma}{\rm \cite{babai}}\label{babaibound} Let $G$ be a graph of order $n$ and let $k$ be such that
$|N(x)\Delta N(y)|\geq k$ for any $x,y\in V(G)$. Then, $G$ has a distinguishing set of cardinality at
most $\lceil\frac{2n\log n}{k+2}\rceil$ provided $k>4\log n$.
\end{lemma}

Note that the graphs considered in this last result are  twin-free. Thus, we deduce from Lemma~\ref{babaibound} and Remark~\ref{sep-loc-dom} another
result supporting Conjecture~\ref{twin_free_conj}: a graph $G$ of order $n\geq 32$ satisfies $\lambda(G)\leq \frac{n}{2}$
whenever  $|N(x)\Delta N(y)|> 4\log n$ for any $x,y\in V(G)$.  Similarly,
we  provide in this section a polynomial time algorithm that produces distinguishing sets of bounded size but  having no restriction
on the twin-free graph $G$.
 Hence, we obtain one of the main results of this paper which is
 an explicit upper bound on $\lambda_{|_{\mathcal{C^*}}}(n)$ (see Subsection~\ref{subsec:ld}).
 Also, this algorithm produces determining sets of bounded size and so an upper bound on the determining number of a twin-free graph (see Subsection~\ref{subsec:det}).
  To do this, we next provide some notation.

 For any set $D\subseteq V(G)$, let us define a relation   on $V(G)$ given by
$u\sim_D v$ if and only if either $u=v$ or $\{u,v\}$ is distinguished by no vertex of $D$.
It is easy to check that this is an equivalence relation, and so we denote by $[u]_D$ 
 the set of vertices $v\in V(G)$ so that $u\sim_D v$. Thus, let
$D^1=\{u\in V(G)\setminus D: |[u]_D|=1\}$ and $D^{>1}=V(G)\setminus (D\cup D^1)$.
Observe that $D,D^1,D^{>1}$ form a partition of $V(G)$, where any of these sets may be empty.
Actually,  $D$ is a distinguishing set if and only if $D^{>1}=\emptyset$.

The following greedy algorithm gives a partition of $V(G)$
into three sets   so that, combining them properly, one obtains distinguishing sets and determining sets
 of $G$ of bounded size, as we shall see in Lemmas~\ref{bb_dis} and \ref{bb_det}.

\begin{algorithm}
\label{algorithm}

\KwIn{A twin-free graph $G$ and a vertex $u_0\in V(G)$.}
\KwOut{An appropriate partition of $V(G)$ into three subsets $A,B,C$.}

\medskip
$A\leftarrow \{u_0\}$\\
$B\leftarrow A^1$\\
$C\leftarrow A^{>1}$\\
\While{$\exists$ $u,x,y\in C$ such that $[x]_A=[y]_A$ and $[x]_{A\cup\{u\}}\neq[y]_{A\cup\{u\}}$}
      {$A\leftarrow A\cup \{u\}$\\
       $B\leftarrow A^1$\\
       $C\leftarrow A^{>1}$\\
      }
\Return{A,B,C}
\caption{}
\end{algorithm}

\subsection{A better upper bound on $\lambda_{|_{\mathcal C^*}}(n)$}\label{subsec:ld}

 Colbourn et al.~\cite{colbourn} showed that the problem of computing the locating-domination number of an arbitrary  graph is   NP-hard.
   Hence, when designing polynomial time algorithms
for computing this parameter, it is necessary the restriction to specific families of graphs. Indeed,
linearity for trees and series-parallel graphs was proved in \cite{colbourn}.
  Likewise, for a twin-free graph $G$, we next show that Algorithm~\ref{algorithm} returns
distinguishing sets of bounded size in polynomial time. Further, this  gives a bound on $\lambda_{|_{\mathcal C^*}}(n)$
which improves that given in Corollary~\ref{erdos_coro}, and consequently the upper bound of Proposition~\ref{lowerCAGAPUSE} by Cáceres et al.~\cite{cagapuse}.

\begin{lemma} \label{bb_dis}
Let $A,B,C$ be the sets obtained by application of Algorithm~\ref{algorithm} to a twin-free graph $G$ and any vertex $u_0\in V(G)$.
Then, $ A\cup B$, $ A\cup C$ and $B\cup C$ are distinguishing sets of $G$.
\end{lemma}

\begin{proof}
 Note first that Algorithm~\ref{algorithm} returns a partition of $V(G)$ into
three subsets $A,B,C$ such that
$B=A^1$, $C=A^{>1}$ and no pair $\{x,y\}\subseteq C$ with $[x]_A=[y]_A$ is distinguished by any $u\in C\setminus\{x,y\}$.   $G$ is twin-free and so
 there must be a vertex $u\in V(G)\setminus\{x,y\}$ distinguishing $\{x,y\}$, which implies that $u\in A\cup B$.
 Hence, $A\cup B$ is a distinguishing set of $G$ since every pair $\{x,y\}\subseteq C$ is distinguished by some $u\in B$ whenever $[x]_A= [y]_A$
(otherwise $[x]_A\neq[y]_A$ and so $\{x,y\}$ is distinguished by some $u\in A$).
Also, $A\cup C$ is a distinguishing set since  every $x\in V(G)\setminus (A\cup C)=B$ is uniquely determined by $A\subseteq A\cup C$.
Finally, to prove that $B\cup C$ is a distinguishing set, let $x_1,...,x_{|A|}$ be the
elements of $A$ sorted by appearance in Algorithm~\ref{algorithm}.
We shall prove that every pair $\{x_i,x_j\}$ with $i<j$ is distinguished by some vertex of $B\cup C$.

In the $i$-th step of the algorithm, a vertex $u\in C$ is added to $A$ and becomes $x_i$
 because $u$ distinguishes a pair $\{x,y\}\subseteq C$ such that $[x]_A=[y]_A$ and
 so this class is splat into two new classes
 $[x]_{A\cup\{x_i\}}$ and $[y]_{A\cup\{x_i\}}$.
 Thus, any pair  $\{\alpha,\beta\}$ with $\alpha\in [x]_{A\cup\{x_i\}}$ and $\beta\in [y]_{A\cup\{x_i\}}$ is
distinguished by $x_i$ and non-distinguished by any of $x_1,...,x_{i-1}$.
 Moreover, in the following steps, it always remains one such pair $\{\alpha,\beta\}\subseteq B\cup C$.
Indeed, when $u$ is sent to $A$, there is another
vertex $u'\in [u]_A$ which either
stays in $C$ or goes to $B$ since $C$ is formed by vertices of non-unitary classes.
 It follows that for every pair $\{x_j,x_k\}$ with $j<k$, there exists
 a pair $\{\alpha,\beta\}\subseteq B \cup C$ non-distinguished by $x_j$
 but distinguished by $x_k$. Thus, assume without loos of generality that
  $x_j\in N(\alpha)\cap N(\beta)$ and $x_k\in N(\alpha)\setminus N(\beta)$ (the remaining cases are analogous).
 Hence, $\{x_j,x_k\}$ is distinguished by $\beta$, which completes the proof.
\end{proof}

The pigeonhole principle ensures that one set among $A,B,C$ has cardinality at least $\lceil\frac{n}{3}\rceil$ and so one of
$A\cup B,A\cup C,B\cup C$ has cardinality at most $\lfloor\frac{2}{3}n\rfloor$. Then, by Lemma~\ref{bb_dis} and Remark~\ref{sep-loc-dom}, we have the following result.

\begin{theorem}\label{23nbound}Let $G$ be a twin-free graph of order $n\geq 4$. Then, there exists   a
 locating-dominating set of $G$ of cardinality at most
$\lfloor\frac{2}{3}n\rfloor+1$ which can be computed in polynomial time. In particular,
 $$\lambda_{|_{\mathcal C^*}}(n)\leq\lfloor\frac{2}{3}n\rfloor+1.$$
\end{theorem}

The next corollary summarizes some of the main results of this paper, i.e.,
Theorems~\ref{lowerbound}, \ref{lambda-deten} and \ref{23nbound}. As far as we know, these are the
best bounds   on the function $(\dim-\dete)(n)$.
\begin{corollary}\label{bestbounds} For every $n\geq 14$,
$$\lfloor\frac{n}{2}\rfloor-1\leq (\dim-\dete)(n)\leq (\lambda-\dete)(n)\leq\lambda_{|_{\mathcal C^*}}(n)\leq\lfloor\frac{2}{3}n\rfloor+1.$$
\end{corollary}

\subsection{An upper bound on $\dete(G)$ for   twin-free graphs}\label{subsec:det}

Blaha~\cite{blaha} showed that finding a minimum base of a permutation group is NP-hard and
provided a greedy algorithm  for constructing bases. The same algorithm was given by Gibbons
 and Laison~\cite{gibbons} in the particular case of
automorphism groups of graphs: for a graph $G$ of order $n$, the algorithm returns a determining
set of size $O(\dete(G)\log\log n)$.
 Observe that this algorithm does not yield a bound
on the determining number of $G$ in terms of $n$. However, we next show that Algorithm~\ref{algorithm} gives an explicit upper bound
on $\dete(G)$ when $G$ is twin-free by constructing a determining set of bounded size  in polynomial time.

\begin{lemma} \label{bb_det}
Let $A,B,C$ be the sets obtained by application of Algorithm~\ref{algorithm} to a twin-free graph $G$ and any vertex $u_0\in V(G)$.
Then,   $A$ and $B\cup C$ are determining sets of $G$.

\end{lemma}

\begin{proof}
 We have proved in Lemma~\ref{bb_dis} that $B\cup C$ is a distinguishing set of $G$,
which implies that it is also a resolving set and so it is a determining set. To prove
 that $A$ is a determining set of $G$, we first claim that $stab(A)=stab(A\cup \{x\})$ for every $x\in B$.
Indeed, $stab(A) \supseteq stab(A\cup \{x\})$, by definition of the stabilizer.
Also, note that $N(x)\cap A$ is unique since $B=A^1$, and recall that automorphisms preserve adjacencies. Thus,
no automorphism fixing every vertex of $A$ can interchange $x$ with any other vertex of $V(G)$.
 Hence,
$stab(A)\subseteq stab (A\cup \{x\})$. Therefore, extending this argument to every vertex of $B$, we obtain that $stab(A)=stab( A\cup B)$. But $ A\cup B$ is a distinguishing set by Lemma~\ref{bb_dis} and so it is a determining set, which implies that $stab( A\cup B)=stab(A)=\{id_G\}$.
It follows that $A$ is a determining set.
\end{proof}

Reasoning as in the previous subsection, we have that either $A$ or $B\cup C$ has cardinality at most $\lfloor\frac{n}{2}\rfloor$ and so,
by Lemma~\ref{bb_det}, we obtain the following bound.

\begin{theorem}\label{twinfreedete}
Let $G$ be a twin-free graph of order $n\geq 4$. Then, there exists a determining set of $G$ of cardinality at most $\lfloor\frac{n}{2}\rfloor$
 which can be computed in polynomial time. In particular,
$$\dete (G)\leq \lfloor\frac{n}{2}\rfloor.$$
\end{theorem}

Note that, although the graph depicted in Figure~\ref{twinfreedet} does not prove tightness for this last result,
this construction shows that we are very close to a tight bound.
\\

\begin{figure}[ht]
\begin{center}
\includegraphics[width=50mm]{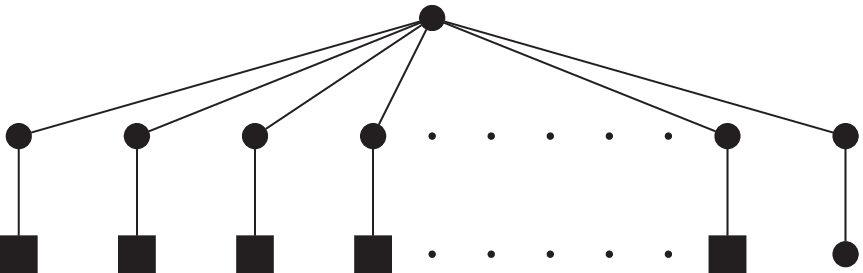}
\caption{A graph $G$ with a  determining set (illustrated with square vertices) of size $\dete(G)=\lfloor\frac{|V(G)|}{2}\rfloor -1$.}\label{twinfreedet}
\end{center}
\end{figure}

\section{Useful tools for the problems restricted to graphs without $K_{2,k}$ as subgraph}\label{sec:C4}

\subsection{$k$-domination}

The concept of $k$-dominating set was introduced by Fink and Jacobson \cite{fink}
 as a generalization of classical dominating sets of graphs. There is a wealth of literature
 about this variety of domination (see \cite{hansberg} and the references given there). Specifically,
 the $k$-domination number $\gamma_k(G)$ has been related to other     graph  parameters
 such as  the path covering number~\cite{graffiti}, the order and the minimum degree~\cite{favaron} and
 the $j$-dependence number~\cite{fink}.
  In this subsection, we establish a relationship between $\gamma_k(G)$  and $\lambda (G)$ when $G$ does not contain
$K_{2,k}$ as a subgraph.
 To do this, we require Lemma~\ref{kdom} below which, in addition, is a key result in the following subsection for
 computing  the restriction of $(\lambda-\dete)(n)$ to this class of graphs   when $k=2$.

Given a set $D\subseteq V(G)$, a vertex $x\in V(G)\setminus D$ and a positive integer $k$, we say that $x$ is {\em $k$-dominated} by $D$ if
 $|N(x)\cap D|\geq k$, and $D$ is a $k$-{\em dominating set} of $G$ if every vertex of $V(G)\setminus D$ is $k$-dominated by $D$.
The $k$-{\em domination number} of $G$, denoted by $\gamma _k(G)$, is the minimum cardinality of a $k$-dominating set of $G$.
 It is straightforward that $\gamma_1(G)=\gamma (G)$ and $\gamma_k (G)\leq\gamma_{\ell}(G)$ for every $k,\ell$ with $1\leq k\leq\ell$.

 Let $\mathcal{K}_{2,k}$ denote  the class of graphs not containing $K_{2,k}$ as a (not necessarily induced) subgraph.
 The following lemma contains the main idea for proving Proposition~\ref{K2k}.

\begin{lemma} \label{kdom} Let $G\in \mathcal{K}_{2,k}$, $D\subseteq V(G)$ and $x\in V(G)\setminus D$.
If $x$ is $k$-dominated by $D$ then, for every $y\in V(G)\setminus D$, the pair $\{x,y\}$ is
distinguished by some vertex of $D$.

\end{lemma}
\begin{proof}
Let $y\in V(G)\setminus D$ and $A\subseteq N(x)\cap D$ such that $|A|=k$.
 Clearly, some vertex of $A$ distinguishes $\{x,y\}$ since otherwise $A\subseteq N(y)$ and so the
 induced subgraph by $A\cup\{x,y\}$ contains a copy
 of $K_{2,k}$, which is impossible.
\end{proof}

Hence, a $k$-dominating set of $G$ is a locating-dominating set whenever $G\in\mathcal{K}_{2,k}$ but the converse is not
true in general, as shown in Figure~\ref{non2dom}. Further,
it was proved in \cite{cockayne} that $\gamma _k(G)\leq \frac{k}{k+1}n$ for any graph $G$ such that $k\leq \delta (G)$.
Thus, we have the following result.

\begin{proposition}\label{K2k} For every $G\in\mathcal{K}_{2,k}$, it holds that $$\gamma (G)\leq \lambda (G)\leq\gamma _k (G).$$ In particular,
$\lambda (G)\leq\frac{k}{k+1}n$ whenever $\delta (G)\geq k$.
\end{proposition}

\begin{figure}[ht]
\begin{center}
\includegraphics[width=50mm]{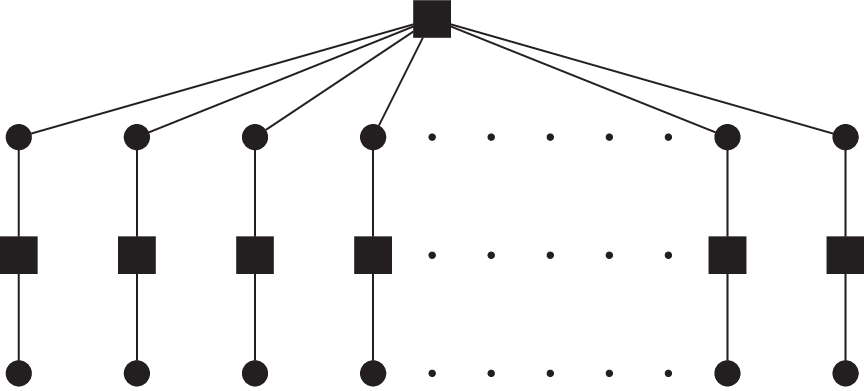}
\caption{A graph in $\mathcal{K}_{2,2}$ with a locating-dominating set (depicted as square vertices) which is not a 2-dominating set.}\label{non2dom}
\end{center}
\end{figure}

Let us denote by $\mathcal{C}_4$ the class $\mathcal{K}_{2,2}$, i.e., the set of graphs not having $C_4=K_{2,2}$ as a subgraph.
 The following corollary  is a consequence  of Proposition~\ref{K2k} when $k=2$, and gives essentially the
 same bound as the one provided by Corollary~\ref{bestbounds}. However, this bound will be improved in the
 following subsection (see Theorem~\ref{C4lambda-det}).

\begin{corollary}
Let $G\in\mathcal{C}_4$ be such that $ \delta (G)\geq 2$. Then, $\lambda (G) -\dete(G)\leq\frac{2}{3}n$.
\end{corollary}

\subsection{Matchings}

The matching number $\alpha'(G)$ has been related to many domination parameters (see for instance \cite{bollobas,seriesB,Hnote,H}).
As an example,  Henning et al.~\cite{Hnote} related the matching number to the {\em total domination number} $\gamma_t(G)$, i.e., the minimum size of a set of vertices dominating every vertex of $G$. 
  Concretely, they proved that $\gamma_t(G)\leq \alpha'(G)$
   whenever $G$ is either a claw-free graph or a $k$-regular graph
    with $k\geq 3$.
In the same vein, we obtain a similar relationship between  $\alpha'(G)$ and   $\lambda (G)$ when $G$ is a twin-free graph in $\mathcal{C}_4$ (see Proposition~\ref{lambdamatching}). Besides its independent interest,
we apply this relation to study the functions
$({\rm dim}-{\rm Det})_{|_{\mathcal{C}_4}}(n)$ and $(\lambda-{\rm Det})_{|_{\mathcal{C}_4}} (n)$ (see Theorems~\ref{C4lambda-det} and
\ref{C4beta-det}).

A {\em matching} $M$ in a graph $G$ is a subset of pairwise disjoint edges of $G$, and the {\em matching number} of $G$, written as
$\alpha '(G)$, is the cardinality of a maximum matching in $G$.
 We denote by $\overline{M}$ the set of vertices of $G$ in no edge of $M$. Observe that $\overline{M}$ is an independent set  when $M$ is maximum (otherwise there is an edge $e=\{x,y\}$ with  $x,y\in \overline{M}$ and so the matching $M'=M\cup \{e\}$ has more edges than $M$, which is impossible). The following is a technical lemma that captures all possible
  situations for the edges of a maximum matching (see Figure~\ref{matching_lemma}).

\begin{lemma}\label{neighbors} Let $M$ be a maximum matching in $G$. Then, for every $\{u,v\}\in M$, exactly one of the following cases holds:
\begin{enumerate}
\item $N(u)\cap \overline{M}=N(v)\cap \overline{M}=\emptyset$.
\item Either $N(u)\cap \overline{M}\neq \emptyset$ or $N(v)\cap \overline{M}\neq\emptyset$, but not both.
\item $N(u)\cap \overline{M}=N(v)\cap \overline{M}=\{x\}$ for some $x\in \overline{M}$.
\end{enumerate}
\end{lemma}
\begin{proof}
Let $M$ be a maximum matching in $G$. It is enough to prove that there is no edge $e=\{u,v\}$ in $M$ and
vertices $x,y\in \overline M$   such that
 $x\in N(u)$ and $y\in N(v)$. Indeed, $(M\setminus\{e\})\cup \{\{u,x\},\{v,y\}\}$ would be
  a matching in $G$ with more edges than $M$, which is impossible.
\end{proof}

\begin{figure}[ht]
\begin{center}
\includegraphics[width=70mm]{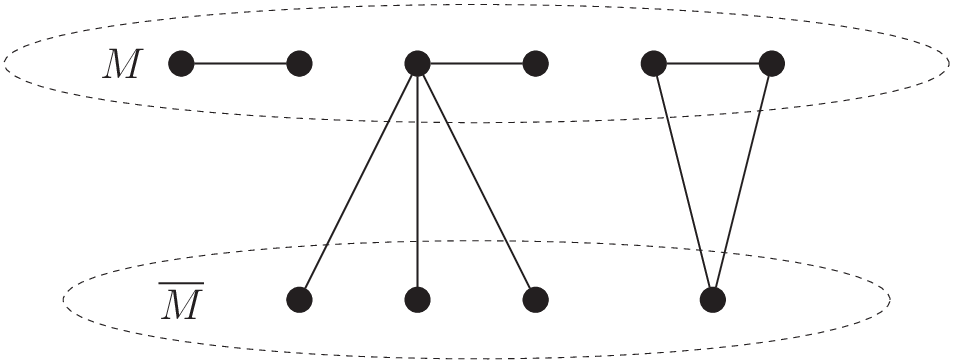}
\caption{The three cases for the edges of a maximum matching $M$ provided by Lemma~\ref{neighbors}.}\label{matching_lemma}
\end{center}
\end{figure}

  For every matching $M$ in $G$, let us consider the set
 $$U_M=\{x\in\overline{M}:N(x)\subseteq e \quad for\quad some \quad e\in M\}.$$
 Note that, when $M$ is maximum,  $U_M$ is formed by all vertices $x\in\overline{M}$ such that $\delta(x)=1$ or $N(x)=e$ for some $e\in M$.
  We next show another technical result which is required in the proof of Proposition~\ref{lambdamatching}.

\begin{lemma}\label{UMempty}
Let $G$ be a twin-free graph. Then, there exists a maximum matching $M$ such that $U_{M}=\emptyset$ which can be computed in polynomial time.
\end{lemma}

\begin{proof}
Let $M$ be a maximum matching in $G$. Observe first that no two vertices $x,y\in U_M$ satisfy
$N(x),N(y)\subseteq e$ for any  $e\in M$ (otherwise Lemma~\ref{neighbors} yields $\delta(x)=\delta(y)=1$ and
$N(x)=N(y)$, which contradicts the fact that $G$ is twin-free).

If $U_M\neq \emptyset$ then let $x\in U_M$ and $e=\{u,v\}$ in $M$ with
$N(x)\subseteq e$. Thus, assuming that $u\in N(x)$, we claim that $M'=(M\setminus\{e\})\cup\{\{u,x\}\}$ is a maximum matching
in $G$ such that $U_{M'}=U_M\setminus\{x\}$. Clearly, $U_M\setminus\{x\}\subseteq U_{M'}$.
 Indeed, for every  $y\in U_M\setminus\{x\}\subseteq \overline M'$ there exists an edge $f\in M$ such that $N(y)\subseteq f$.
 As remarked above,  $f\neq e$ since $x\neq y$, which implies that $f\in (M\setminus e)\subseteq M'$ and so $y\in U_{M'}$.

We now prove that $U_{M'}\subseteq U_M\setminus\{x\}$. Let $y\in U_{M'}$ such that $N(y)\subseteq f$ for some $f\in M'$.
 If $f\in M\setminus\{e\}$ then $y\in U_M\setminus\{x\}$ (note that $y\neq v$ since
$v\in N(u)$ and so there is no $f\in M\setminus\{e\}$ with $N(v)\subseteq f$).
Otherwise, $f=\{u,x\}$. If $y\neq v$ then $y\notin N(x)$ since $\overline M$ is an independent set, and then
$N(y)=\{u\}$. However, $N(x)\subseteq e$ and so Lemma~\ref{neighbors} ensures that $N(x)=N(y)=\{u\}$; a contradiction since $G$ is twin-free.
Therefore, $y=v$ and we easily get either $N(v)=N(x)=\{u\}$ or
$N[v]=N[x]=\{u,v,x\}$; again a contradiction. Thus, we have proved that $U_{M'}=U_M\setminus\{x\}$ and iterating this process gives a maximum matching $M^*$ with $U_{M^*}=\emptyset$. Observe that     $M$ can be found in polynomial time \cite{edmonds} and
  $M^*$ is easily obtained from $M$ also in polynomial time. Hence, we can compute $M^*$ in polynomial time, as claimed.

\end{proof}

We now reach one of the main results of this section which relates $\alpha'(G)$ and $\lambda(G)$ when
$G$ is a twin-free graph in $\mathcal{C}_4$.

\begin{proposition}\label{lambdamatching}
Let $G\in\mathcal{C}_4$ be a twin-free graph of order $n\geq 4$. Then, there is
 a locating-dominating set of $G$ of cardinality $\alpha'(G)$ which can be computed in polynomial time, and consequently
 $$\lambda(G)\leq \alpha '(G).$$
In particular, $\lambda (G)\leq \frac{n}{2}$.
\end{proposition}

\begin{proof}
Let $M$ be a maximum matching in $G$ satisfying that $U_M=\emptyset$, which exists by Lemma~\ref{UMempty}.
We consider a partition $V(G)=V_1\cup V_2\cup \overline M$ such that $e\cap V_1$ and $e\cap V_2$ are non-empty for every $e\in M$, i.e.,
 $V_1$ and $V_2$ contain the endpoints of every $e\in M$, respectively.
 By Lemma~\ref{neighbors}, for every $e=\{u,v\}$ in $M$ and $x\in\overline M$ with $N(x)\cap e=\{u\}$,
 we can assume without loss of generality that  $u\in V_1$.
 This means that every $e=\{u,v\}$ in $M$ so that $N(u)\cap \overline M\neq \emptyset$ and $N(v)\cap M=\emptyset$
 satisfies $u\in V_1$.
Thus, we  shall prove that $V_1$ is a locating-dominating set of $G$.

It is easy to check that $V_1$ is a dominating set of $G$ by construction of $V_1$ and $V_2$.
 Furthermore, every  $x\in\overline M$ is 2-dominated by $V_1$. Indeed, $N(x)$ intersects
at least two different edges of $M$ since $U_M=\emptyset$.
Thus, let $u,u'\in N(x)$ with $\{u,v\},\{u',v'\}\in M$ for some $v,v'\in V(G)$. Since we can suppose $u,u'\in V_1$,
 we have that  $x$ is 2-dominated by $\{u,u'\}\subseteq V_1$.

To prove that $V_1$ is a distinguishing set of $G$, we claim that every pair $\{x,y\}\subseteq V_2\cup \overline M$
 is distinguished by some $u\in V_1$. By Lemma~\ref{kdom}, we
can assume that $x,y\in V_2$ since every vertex of $\overline{M}$
is 2-dominated by $V_1$ and $G\in\mathcal{C}_4$. Thus, let $u,u'\in V_1$ such that $\{u,x\},\{u',y\}\in M$.
Hence, one of $u$ or $u'$ resolves $\{x,y\}$ since otherwise $u\in N(y)$ and $u'\in N(x)$, which produces
 the cycle $(u,x,u',y)$; a contradiction with $G\in\mathcal{C}_4$. Therefore, we have proved that $V_1$ is a locating-dominating set of $G$
 (obtained in polynomial time by Lemma~\ref{UMempty})
  and so $\lambda (G)\leq |V_1|=\alpha'(G)\leq \frac{n}{2}$, as claimed.
\end{proof}

As an application of this last result, we next compute the exact value of $(\lambda-\dete)_{|_{{\mathcal{C}_4}}}(n)$ and
give bounds on $(\dim-\dete)_{|_{{\mathcal{C}_4}}}(n)$, supporting again  the validity of Conjecture~\ref{twin_free_conj}.

\begin{theorem} \label{C4lambda-det}
For every $n\geq14$, it holds that
$$(\lambda-\dete)_{|_{{\mathcal{C}_4}}}(n)=\lfloor\frac{n}{2}\rfloor.$$
\end{theorem}
\begin{proof}

Mimicking the proof of Theorem~\ref{lowerbound} on $(\lambda-\dete)(n)$
 yields  $(\lambda-\dete)_{|_{{\mathcal{C}_4}}}(n)\geq\lfloor\frac{n}{2}\rfloor$
  since the graphs considered   belong to $\mathcal{C}_4$. To prove the reverse inequality,
  it suffices to show that every graph $G\in \mathcal{C}_4$ of order $n$ satisfies $\lambda(G)-\dete(G)\leq \frac{n}{2}$.
Indeed, let us assume first that $G^*\not\cong K_2$ (otherwise $\lambda(G)-\dete(G)\leq 1<\frac{n}{2}$ by
 Lemma~\ref{detn-r}). Thus, the graph $\widetilde{G}$ described
in Section~\ref{sec:twin-free} satisfies $\lambda(G)-\dete(G)\leq \lambda (\widetilde{G})$, by Proposition~\ref{lambdatilde}.
  Also, Lemma~\ref{Gtilde} guarantees that $\widetilde G$ is twin-free and has order $\widetilde n\leq n$.
  Hence, Proposition~\ref{lambdamatching} gives $\lambda (\widetilde{G})\leq\frac{\widetilde{n}}{2}\leq \frac{n}{2}$
  since it is easily seen that $\widetilde G\in\mathcal{C}_4$.
  Therefore,  we have proved that  $\lambda(G)-\dete(G)\leq \lambda (\widetilde{G})\leq\frac{n}{2}$, as required.
\end{proof}

\begin{theorem} \label{C4beta-det}
For every $n\geq 49$, it holds that
$$\lfloor\frac{2}{7}n\rfloor\leq(\dim-\dete)_{|_{{\mathcal{C}_4}}}(n)\leq \lfloor\frac{n}{2}\rfloor.$$
\end{theorem}
\begin{proof}
The upper bound follows immediately from Expression (\ref{original}) and Theorem~\ref{C4lambda-det}. For the lower bound, we shall construct a graph  $G$ of order $n$ not containing $C_4$ as a subgraph such that $ (\dim-\dete)(G)=\lfloor\frac{2}{7}n\rfloor$.
Indeed, let $n=7q+s$ for some integers $q,s$ with $q\geq 7$ and $0\leq s<7$. The graph $T_{q,0}$ is given by attaching a copy
of $T_6$ to every vertex of $T_{q-1}$ as
 shown in Figure~\ref{Tqs}(a) (recall that  the tree $T_r$ is described in Section~\ref{sec:lowerbound}). If $s\in\{1,2,3\}$ then $T_{q,s}$ is
obtained from $T_{q,0}$ by replacing the edge $\{u_1,u_2\}$ by a path of length $s+1$ (see Figure~\ref{Tqs}(b)). Otherwise,
$s\in\{4,5,6\}$ and $T_{q,s}$ comes from $T_{q,0}$  by attaching a path of length $s$ to $u_1$ (see Figure~\ref{Tqs}(c)).
 It is clear that $aut(T_{q,s})=id_{T_{q,s}}$, which implies that $\dete (T_{q,s})=0$. Further, the metric bases illustrated in
Figure~\ref{Tqs} (see Section~\ref{sec:trees} for more information on metric dimension of trees) show that
$$\dim (T_{q,s})= \left\{ \begin{array}{lcc}
                      2q  &  if & s\in\{0,1,2,3\} \\
                     \\ 2q+1 &   if  & s\in\{4,5,6\}
             \end{array}
   \right.$$
But  $n=|V(T_{q,s})|$ and so $\dim (T_{q,s})=\lfloor\frac{2}{7}n\rfloor$.
Therefore, setting $G\cong T_{q,s}$ yields the expected bound.

\end{proof}

\begin{figure}[ht]
\begin{center}

\begin{tabular}{ccc}
\includegraphics[width=70mm]{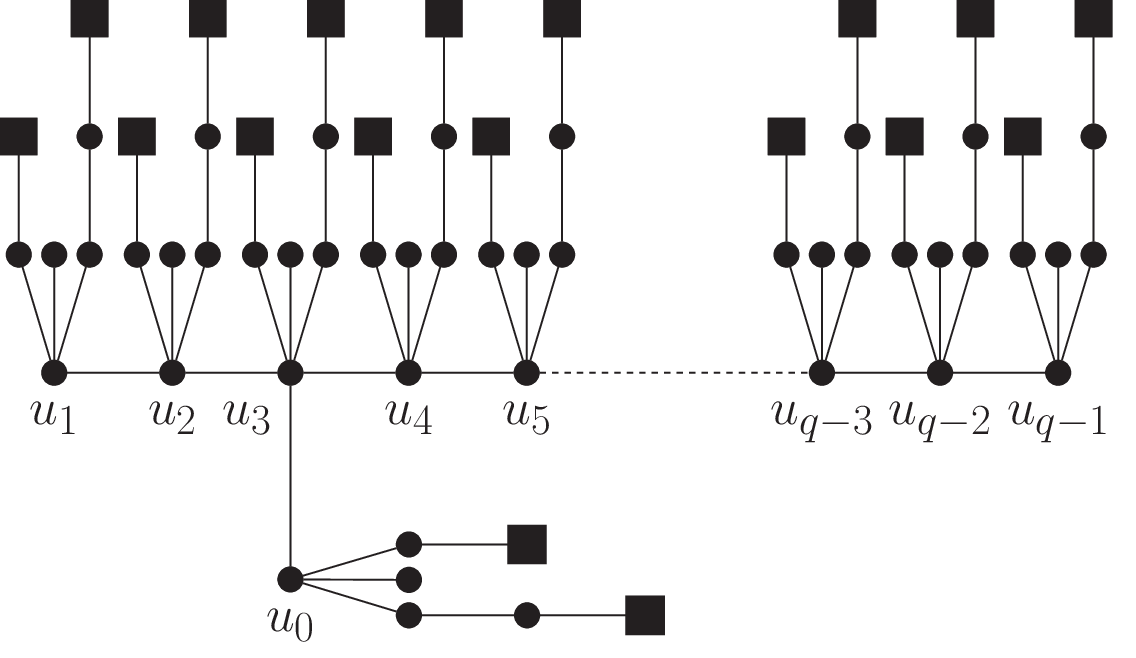}
 &&
\includegraphics[width=77mm]{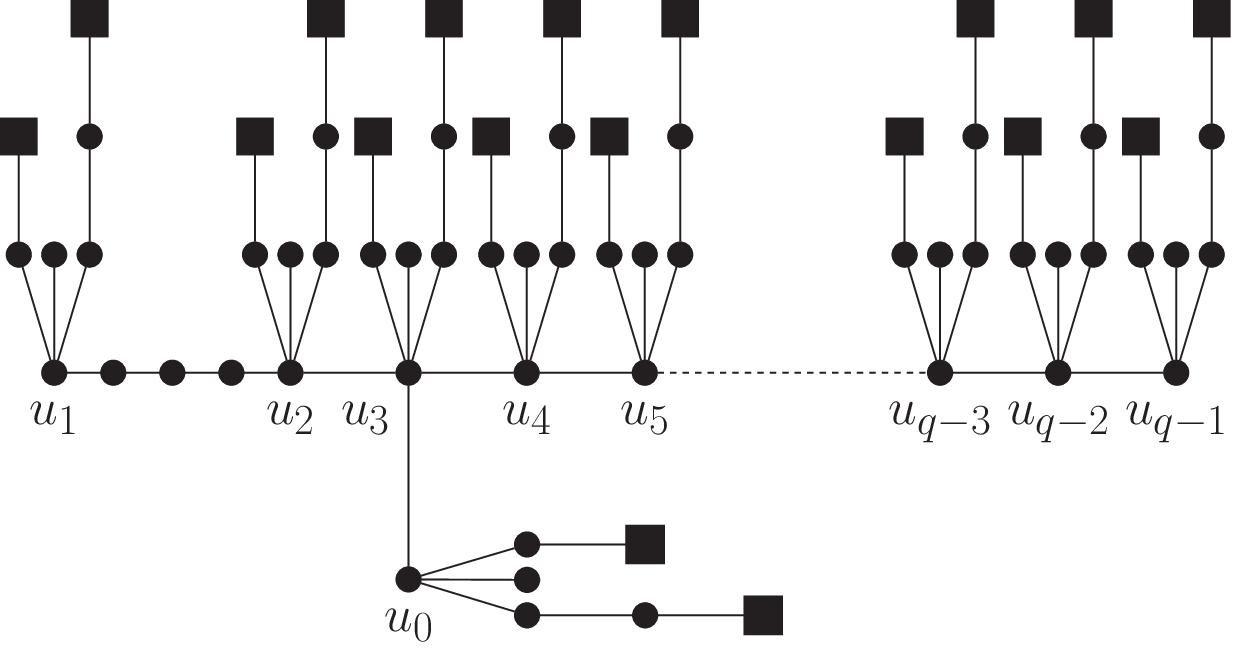}\\
(a)&  & (b)  \\
\\
\end{tabular}

\begin{tabular}{c}

\includegraphics[width=77mm]{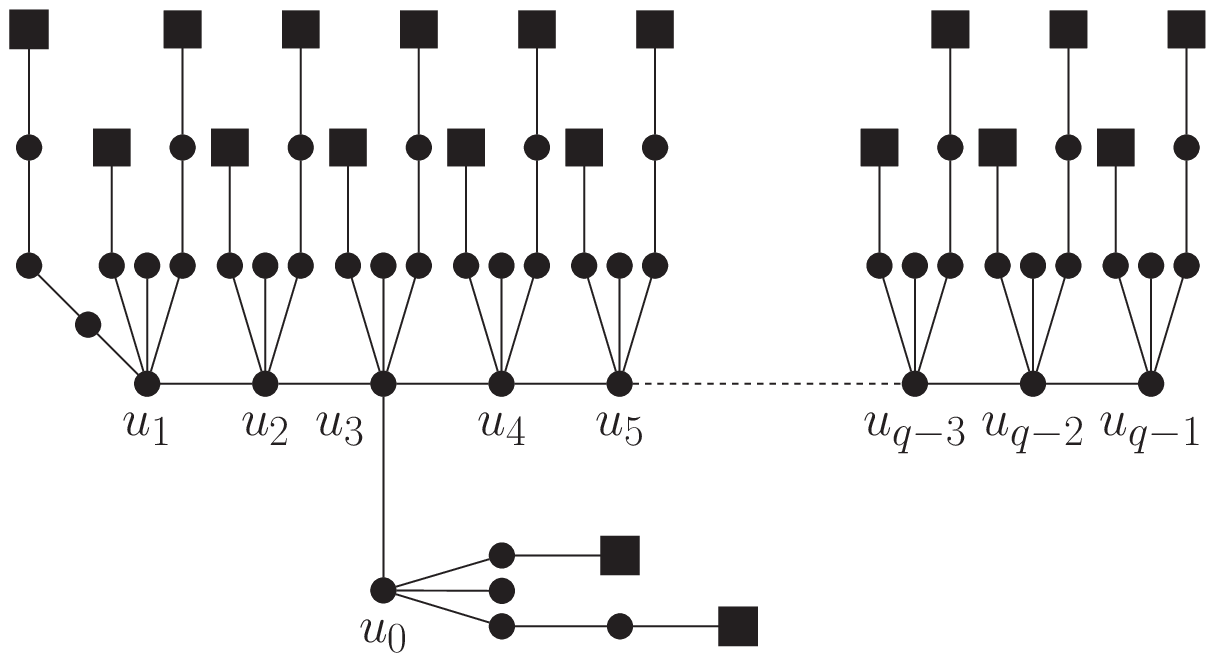}\\

   (c)     \\

\end{tabular}
\end{center}
\caption{Metric bases (depicted as square vertices) of the graphs (a) $T_{q,0}$, (b) $T_{q,3}$ and (c) $T_{q,5}$.}\label{Tqs}
\end{figure}

\section{Computing the functions restricted to trees}\label{sec:trees}

Cáceres et al.~\cite{cagapuse} started the study of the difference between the determining number and the metric dimension
 of trees when trying to answer the question raised by Boutin~\cite{boutin}. Actually,
they constructed a family of trees where this difference is $\Omega(\sqrt n)$.
This section completely solves this particular problem showing
 that the   trees $T_{q,s}$ described in the preceding section have  the maximum value
of $\dim(G)-\dete (G)$ in the class of trees (see Theorem~\ref{beta-detetrees}). Moreover, we compute the
  maximum value of $\lambda(G)-\dete(G)$ restricted to trees (see Theorem~\ref{lambda-detetrees}).

Some  of the terminology that we adopt in this section can be found in \cite{chartrand}.
Given a tree $T$, a vertex of degree at least 3 is called a {\it major vertex} of $T$.
A pendant vertex $\ell$ is a {\it terminal vertex} of a major vertex $u$ if
the major vertex closest to $\ell$ in $T$ is $u$.
 The {\it terminal degree} of a major vertex $u$, denoted by $ter(u)$, is the number of terminal vertices of $u$.
 A major vertex $u$ is an
{\it exterior major vertex} of $T$ if it has positive terminal degree in $T$.
The set of exterior major vertices of $T$ is denoted by $Ex(T)$.

 The metric dimension of any tree is well-known
(see for instance \cite{chartrand,melter,landmarks,slater}) and its formula is exhibited next.

\begin{proposition}{\rm \cite{landmarks}}\label{betatrees}
If $T$ is a tree that is not a path, then
$$\dim (T)=\sum_{u\in Ex(T)} (ter(u)-1).$$
\end{proposition}

First, we provide two technical lemmas which aid in proving Theorem~\ref{beta-detetrees}.
We denote by $ter'(u)$ the  number of different distances between $u$ and any of its terminal vertices.
For every $u\in Ex(T)$,  we write $N_u$  for the set of vertices in some $u$-$\ell$ shortest path, where $\ell$ is
a terminal vertex of $u$; the cardinality of $N_u$ is denoted by $n_u$.

\begin{lemma}\label{ter'}
Let $T$ be a tree and $u\in Ex(T)$. Then, $ter'(u)\leq \frac{2}{7}n_u+1$.
\end{lemma}

\begin{proof}
Let $d_1,d_2,...,d_{ter'(u)}$ with $d_1<d_2<...<d_{ter'(u)}$ be the different distances between $u$ and any of its terminal vertices.
Thus, $\displaystyle n_u\geq (\sum_{i=1}^{ter'(u)}d_i)+1$, and consequently
$$n_u\geq (\sum_{i=1}^{ter'(u)}i)+1=\frac{ter'(u)(ter'(u)+1)}{2}+1.$$
Hence, an easy computation shows that $ter'(u)\leq \frac{\sqrt{8n_u-7}-1}{2}\leq \frac{2}{7}n_u+1$.
\end{proof}

\begin{lemma}\label{detetrees}
Let $T$ be  a tree that is not a path. Then,  $$\dete (T)\geq \sum_{u\in Ex(T)}(ter(u)-ter'(u)).$$

\end{lemma}

\begin{proof}
Let $S$ be a minimum determining set of $T$. As shown in \cite{erwin}, we can assume that $S$ is only formed
by pendant vertices. Consider a vertex $u\in Ex(T)$ and two of its terminal vertices, say $\ell$ and $\ell'$, such that $d(u,\ell)=d(u,\ell')$.
Clearly, either $\ell$ or $\ell'$ belongs to $S$ (otherwise there is an automorphism interchanging the $u-\ell$ path and the $u-\ell'$ path,
and fixing the remaining vertices of $V(T)$  but $S$ is a determining set; a contradiction).
 Therefore, at least
$ter(u)-ter'(u)$ vertices of $N_u$ are in $S$, and extending this argument to $Ex(T)$ yields the bound.
\end{proof}

We now achieve one of the main results of this section which provides the exact value of
the function $(\dim - \dete)_{|_{\mathcal{T}}} (n)$, where $\mathcal{T}$ denotes the family  of trees.

\begin{theorem}\label{beta-detetrees}
For every $n\geq 49$, it holds that
$$(\dim - \dete)_{|_{\mathcal{T}}} (n)=\lfloor\frac{2}{7}n\rfloor.$$

\end{theorem}

\begin{proof} We first prove that $(\dim - \dete) (T)\leq \lfloor\frac{2}{7}n\rfloor$ for any tree $T$ of order $n$.
Thus, we can assume that $T$ is not a path since it is clear that $\dim(P_n)-\dete(P_n)=1-1=0$ for every $n\geq 2$.
By Proposition~\ref{betatrees},
$$
\begin{array}{ccl}
\dim (T) & = &\displaystyle \sum_{u\in Ex(T)} (ter(u)-1) \\
           & = &\displaystyle \sum_{u\in Ex(T)} (ter(u)-ter'(u))+\sum_{u\in Ex(T)} (ter'(u)-1).
\end{array}
$$
Hence, according to Lemma~\ref{detetrees}, we get
$$\dim (T) - \dete (T)\leq     \sum_{u\in Ex(T)} (ter'(u)-1)     \leq\frac{2}{7}n,$$
 the last inequality being a consequence of Lemma~\ref{ter'}.
This shows that $(\dim-\dete)_{|_{\mathcal{T}}}(n)\leq\lfloor\frac{2}{7}n\rfloor$ and equality is given by the
 graphs $T_{q,s}$ constructed in the proof of Theorem~\ref{C4beta-det}, which are trees.
\end{proof}

We want to stress that this last result ensures that trees are not the appropriate family
of graphs for disproving Conjecture~\ref{centralconj} and shows how far is the bound $\Omega(\sqrt n)$ due to Cáceres et al.~\cite{cagapuse}.

Since trees  do not contain $C_4$ as a subgraph, i.e., $\mathcal{T}\subset\mathcal{C}_4$, then
$(\lambda-\dete)_{|_{\mathcal{T}}}(n)\leq(\lambda-\dete)_{|_{\mathcal{C}_4}}(n)=\lfloor\frac{n}{2}\rfloor$, by Theorem~\ref{C4lambda-det}.
Further,  the graphs in the proof of Theorem~\ref{lowerbound} are trees and so we   get the following result.

\begin{theorem}\label{lambda-detetrees}
For every $n\geq 14$, it holds that
$$(\lambda-\dete)_{|_{\mathcal{T}}}(n)=\lfloor\frac{n}{2}\rfloor.$$
\end{theorem}

\section{Concluding remarks and open questions}\label{sec:concluding}

In this paper, we have studied the function $(\dim-\dete)(n)$ for which we have developed
an independent study on    $(\lambda -\dete)(n)$ and $\lambda_{|_{\mathcal C^*}}(n)$.
Thus, we provide lower and upper bounds on these functions  which in particular
improve those given by Cáceres et al.~\cite{cagapuse} for $(\dim-\dete)(n)$.
To do this, we construct two appropriate families of graphs for improving  the lower bound.
For the upper bound, we develop a technique which uses locating-dominating sets as a main tool.
Indeed, we show that $(\dim-\dete)(n)$ and $(\lambda -\dete)(n)$ are bounded above by
the function $\lambda_{|_{\mathcal C^*}}(n)$. To obtain bounds on this function,
we first provide a variant of a well-known   theorem by Ore~\cite{ore} which implies
a number of consequences between the locating-domination number and other graph parameters.
One of these consequences yields a first upper bound on $\lambda_{|_{\mathcal C^*}}(n)$ by means of a classical result
 due to Erd\H{o}s and Szekeres~\cite{szekeres}.

 The second upper bound on  $\lambda_{|_{\mathcal C^*}}(n)$ comes from the designing of a polynomial time algorithm
 that produces both distinguishing sets and determining sets of twin-free graphs. Thus, we also obtain a bound on the determining number
 of a twin-free graph.

 Finally, we restrict ourselves to graphs not having $K_{2,k}$ as a subgraph, thus
  relating  the locating-domination number to the $k$-domination number and
 the matching number. These relations
 produce bounds and exact values of the restrictions of   $(\dim-\dete)(n)$ and $(\lambda -\dete)(n)$
 to the graphs without $C_4$ as a subgraph and the class of trees.
  Specifically, we solve the problem
 first considered   by Cáceres et al.~\cite{cagapuse} about the difference between the determining number
 and the metric dimension of a tree.

 It would be interesting to settle Conjectures \ref{centralconj} and \ref{twin_free_conj}, which predict  the exact values of
the functions $(\dim-\dete)(n)$, $(\lambda -\dete)(n)$ and $\lambda_{|_{\mathcal C^*}}(n)$ . Also,
it remains open the computation of the function $({\rm dim}-{\rm Det})_{|_{\mathcal{C}_4}}(n)$.
 Further, it would be also of interest to find particular families of graphs where the restrictions of
   $(\dim-\dete)(n)$ and $(\lambda -\dete)(n)$ may be computed. Finally, the maximum value of
   the difference between the metric dimension and the locating-domination number is still unknown
   and a study on this function may be proposed.

\bibliographystyle{plain}

\bibliography{biblio_BSGRS}

\end{document}